\documentclass[11pt]{amsart}
\usepackage{amssymb,latexsym,amsmath,amsfonts}
\usepackage[mathscr]{eucal}
\numberwithin{equation}{section}
\theoremstyle{definition}
\newtheorem{definition}{Definition}[section]
\newtheorem{remark}[definition]{Remark}

\newtheorem{custom}[definition]{{}}
\theoremstyle{plain}
\newtheorem{theorem}[definition]{Theorem}
\newtheorem{maintheorem}[definition]{Main Theorem}
\newtheorem{lemma}[definition]{Lemma}
\newtheorem{proposition}[definition]{Proposition}
\newtheorem{result}[definition]{Result}

\newcommand{\CC}{\mathbb{C}^2}

\newcommand{\Cn}{\mathbb{C}^n}
\newcommand{\cplx}{\mathbb{C}}
\newcommand{\RR}{\mathbb{R}}
\newcommand{\tor}{\mathbb{T}}

\newcommand{\bdy}{\partial}

\newcommand{\OM}{\Omega}

\newcommand{\Dsc}{\mathbb{D}}
\newcommand\scone[2]{\boldsymbol{\mathcal{K}}({#1};{#2})}
\newcommand{\clcone}{\boldsymbol{\overline{\mathcal{K}}}}
\newcommand{\clBcone}{\boldsymbol{\overline{\mathcal{K}}_*}}
\newcommand\smcone[2]{\overline{\boldsymbol{\mathcal{K}}({#1}};{#2})}
\newcommand{\weg}{\boldsymbol{\mathscr{W}}}
\newcommand{\clweg}{\boldsymbol{\overline{\mathscr{W}}}}

\newcommand{\smoo}{\mathcal{C}}
\newcommand{\hol}{\mathcal{O}}

\newcommand{\psh}{{\sf psh}}

\newcommand{\eps}{\varepsilon}

\newcommand{\eh}{\varkappa}

\newcommand{\zt}{\zeta}
\newcommand{\zbar}{\overline{z}}
\newcommand{\wbar}{\overline{w}}
\newcommand{\ztbar}{\overline{\zeta}}
\newcommand{\tht}{\theta}

\newcommand{\Lam}{\Lambda}
\newcommand{\zahl}{\mathbb{Z}}

\newcommand{\Nat}{\mathbb{N}}
\newcommand{\er}{\mathfrak{Re}}

\newcommand{\rDim}{{\rm dim}_{\mathbb{R}}}

\newcommand{\funz}{{\sf F}(r;}

\newcommand{\hess}{\mathfrak{H}_{\mathbb{C}}}
\newcommand{\cplxlns}{\boldsymbol{{\sf L}}}
\newcommand{\exepc}{\boldsymbol{\mathcal{E}}}
\newcommand{\dil}{\mathscr{D}}
\newcommand{\ev}{\mu}
\newcommand{\zerov}{\boldsymbol{0}}
\newcommand{\indx}{\mathscr{I}}

\newcommand{\Dis}{{\sf dsc}}
\newcommand{\spn}{{\sf span}}
\newcommand{\lrarw}{\longrightarrow}

\newcommand\levi[1]{\mathfrak{L}{#1}}
\newcommand{\cpln}{\varLambda}
\newcommand{\jayP}{\widetilde{{}^j P}}

\newcommand\rg[1]{{\sf Arg}({#1})}

\newcommand\LeviDeg[1]{\omega({#1})}

\newcommand\sig[1]{\widetilde{\sigma_{#1}}}
\newcommand{\lcm}{{\rm lcm}}
\newcommand{\bcdot}{\boldsymbol{\cdot}}
\newcommand{\hyprs}{\boldsymbol{H}}

\newcommand{\dbar}{\overline{\partial}}
\newcommand\mixderiv[3]{\frac{\partial^2{#3}}{\partial{#1}\partial\overline{{#2}}}}
\newcommand\deriv[3]{\frac{\partial^{{#1}}{#2}}{\partial{#3}}}
\newcommand\mixRderiv[3]{\frac{\partial^2{#3}}{\partial{#1}\partial{#2}}}
\newcommand\Mixdrv[2]{\partial^2_{{#1}\overline{{#2}}}}
\newcommand\monodrv[1]{\partial_{{#1}}}

\newcommand{\laplc}{\bigtriangleup}

\begin{document}

\title[Plurisubharmonic polynomials and bumping]{Plurisubharmonic polynomials and bumping} 

\author{Gautam Bharali}
\address{Department of Mathematics, Indian Institute of Science, Bangalore - 560012}
\email{bharali@math.iisc.ernet.in}
\author{Berit Stens{\o}nes}
\address{Department of Mathematics, University of Michigan, 530 Church Stret, Ann Arbor, MI
48109}
\email{berit@umich.edu}
\thanks{The first-named author is supported by a grant from the UGC under DSA-SAP, Phase~IV}
\keywords{Bumping, finite-type domain, plurisubharmonic function, weighted-homogeneous function}
\subjclass[2000]{Primary 32F05, 32T25}

\begin{abstract}
We wish to study the problem of bumping outwards a pseudoconvex,
finite-type domain $\OM\subset\Cn$ in such a way that pseudoconvexity 
is preserved and such that the lowest possible orders of contact of the bumped domain 
with $\bdy\OM$, at the site of the bumping, are explicitly realised. Generally,
when $\OM\subset\cplx^n, \ n\geq 3$, the known methods lead to
bumpings with high orders of contact --- which are not explicitly known
either --- at the site of the bumping. Precise orders are known for 
$h$-extendible/semiregular domains. This paper is motivated by certain 
families of {\em non-semiregular} domains in $\mathbb{C}^3$. These families are 
identified by the behaviour of the least-weight plurisubharmonic polynomial 
in the Catlin normal form. Accordingly, we study how to perturb certain 
homogeneous plurisubharmonic polynomials without destroying plurisubharmonicity.
\end{abstract}

\maketitle

\section{Introduction}\label{S:intro}

This paper is a part of a study of the boundary-geometry of bounded pseudoconvex domains of 
finite type. For such a domain in $\cplx^2$, one can demonstrate many nice properties
that have major function-theoretic consequences. For example, Bedford and Fornaess
\cite{bedfordFornaess:cpfwpd78} showed that every boundary point of 
such a domain with {\em real-analytic}
boundary admits a holomorphic peak function. A similar conclusion was obtained by Fornaess and
Sibony \cite{fornaessSibony:cpfwpd89} in the finite-type case. Later, Fornaess \cite{fornaess:sneDbC286}
and Range \cite{range:ikHeDbpdftC290} exploited a crucial ingredient needed in both 
\cite{bedfordFornaess:cpfwpd78} and \cite{fornaessSibony:cpfwpd89} to obtain H{\"o}lder estimates
for the $\dbar$-problem.
\smallskip

All of the mentioned results depend on the fact that the given domain in $\cplx^2$ has a good
local bumping. If $\OM\subset\Cn, \ n\geq 2$, is a smoothly bounded pseudoconvex domain 
and $\zt\in\bdy\OM$, we say that {\em $\OM$ admits a local bumping around $\zt$} if we can find 
a neighbourhood $U_\zt$ of $\zt$ and a smooth function
$\rho_\zt\in\psh(U_\zt)$ such that
\begin{itemize}
\item[i)] $\rho_\zt^{-1}\{0\}$ is a smooth hypersurface in $U_\zt$ that is pseudoconvex from the
side $U^{-}_\zt:=\{z:\rho_\zt(z)<0\}$; and
\item[ii)] $\rho_\zt(\zt)=0$, but 
$(\overline{\OM}\setminus\{\zt\})\bigcap U_\zt\varsubsetneq U^{-}_\zt$.
\end{itemize}
We shall call the triple $(\bdy\OM,U_\zt,\rho_\zt)$ a local bumping of $\OM$ around 
$\zt$. 
\smallskip

Diederich and Fornaess \cite{diederichFornaess:phmpdrab79} did show that if $\Omega$ is a bounded,
pseudoconvex domain with real-analytic boundary, then local bumpings always exist around each
$\zt\in\bdy\OM$. The problem however --- from the viewpoint of the applications mentioned above --- is
that the order of contact between $\bdy\OM$ and $\rho_\zt^{-1}\{0\}$ at $\zt$ might be very high. In 
fact, if $\OM\subset\Cn, \ n\geq 3$, then this order of contact is much higher, in many cases, than 
the type of the point $\zt\in\bdy\OM$. (See Catlin's \cite{catlin:sedbarNpd87} and D'Angelo's
\cite{d'angelo:rhoca82} for different notions of type.) Given these function-theoretic motivations,
one would like to attempt to solve the following problem:
\begin{itemize}
\item[(*)] With $\OM$ and $\zt\in\bdy\OM$ as above, construct a local bumping 
$(\bdy\OM,U_\zt,\rho_\zt)$ such that the orders of contact of $\bdy\OM\bigcap U_\zt$ with
$\rho_{\zt}^{-1}\{0\}$ at $\zt$ along various directions $V\in T_\zt(\bdy\OM)\bigcap iT_\zt(\bdy\OM)$
are the {\em lowest possible} and {\em explicitly known.}
\end{itemize}
\smallskip

In the absence of any convexity near $\zt\in\bdy\OM$, we are led to consider the following
situation: given $\OM$ and $\zt\in\bdy\OM$ as above, we can find a
local holomorphic coordinate system $(V_\zt;w,z_1,\dots,z_{n-1})$, centered at $\zt$, such that
\[
\OM\bigcap V_\zt \ = \ \left\{(w,z)\in V_\zt:\er(w)+P_{2k}(z)+
O(|z|^{2k+1}, |wz|, |w|^2)<0\right\},
\]
where $P_{2k}$ is a plurisubharmonic polynomial in $\cplx^{n-1}$ that is homogeneous of degree
$2k$. The first result in $\Cn, \ n\geq 3$, to address (*) is due to Noell \cite{Noell:pfpd93}.
He showed that if $P_{2k}$ is plurisubharmonic and is not harmonic along any complex line through
$0$, then $\OM$ can be bumped homogeneously to order $2k$ around $0\in\Cn$. So, one would like to
know whether some form of Noell's result holds {\em without the ``nonharmonicity'' assumption
on $P_{2k}$}. It is worth noting here that if $\bdy\OM$ is either convex or lineally convex near
$\zt$, then this additional property yields a local holomorphic coordinate system
$(V_\zt;w,z_1,\dots,z_{n-1})$ such that (*) is satisfactorily solved. This is the content of
the Diederich-Fornaess papers \cite{diederichFornaess:sfcdft99} and  
\cite{diederichFornaess:lcdft:hsf99}.
\smallskip 

A clearer connection between (*) and the boundary-geometry emerges when we look at the
Catlin normal form for $\bdy\OM$ near a finite-type $\zt\in\bdy\OM$. If $(1,m_1,\dots,m_{n-1})$ is 
the Catlin multitype of $\zt$ (readers are once more referred to \cite{catlin:sedbarNpd87}), and
we write $\Lam:=(m_1,\dots,m_{n-1})$, then, there exists a local holomorphic coordinate
system $(V_\zt;w,z_1,\dots,z_{n-1})$, centered at $\zt$, such that
\begin{multline}\label{E:localdef}
\OM\bigcap V_\zt \\
 = \ \left\{(w,z)\in V_\zt:\er(w)+P(z)+O(|wz|,|w|^2)+(\text{\em higher-weight terms in $z$})<0\right\},
\end{multline}
where $P$ is a $\Lam$-homogeneous plurisubharmonic polynomial in $\cplx^{n-1}$ that has no 
pluriharmonic terms. We say that $P$ is {\em $\Lam$-homogeneous} if
$P(t^{1/m_1}z_1,\dots,t^{1/m_{n-1}}z_n)=tP(z_1,\dots,z_{n-1})$ 
$\forall z=(z_1,\dots,z_{n-1})\in\cplx^{n-1}$ and for every $t>0$. Note that (*) would be completely
solved if one could prove the existence of a $\Lam$-homogeneous
function $H\in\smoo^\infty(\cplx^{n-1})$ satisfying
\[
H(z) \ \geq \ C\sum_{j=1}^{n-1}|z_j|^{m_j} \quad\forall z\in\cplx^{n-1},
\]
for some $C>0$, such that $(P-H)$ is strictly plurisubharmonic on $\cplx^{n-1}\setminus\{0\}$.
Unfortunately, this plan does not work in general. Yu in \cite{yu:pfwpd94}, and 
Diederich and Herbort in \cite{diederichHerbort:pdst94}, have independently shown that
\begin{itemize}
\item[(**)] A $\Lam$-homogeneous $H$ of the sort described above exists only if there are
no complex-analytic subvarieties of $\cplx^{n-1}$ of positive dimension along which $P$ is 
harmonic.
\end{itemize}
\smallskip

There are certainly domains in $\cplx^3$ for which the condition in (**) fails. Our paper is
inspired by the following examples where that condition fails:
\smallskip

\noindent{{\em Example~1:}
\[
\OM_1 \ = \ \left\{(w,z)\in\cplx\times\CC:\er(w)+|z_1|^6|z_2|^2+|z_1|^8+
		\tfrac{15}{7}|z_1|^2\er(z_1^6)+|z_2|^{10}<0\right\},
\]
and}

\noindent{{\em Example~2:}
\[
\OM_2 \ = \ \left\{(w,z)\in\cplx\times\CC:\er(w)+|z_1z_2|^8+
                \tfrac{15}{7}|z_1z_2|^2\er(z_1^6z_2^6)+|z_1|^{18}+|z_2|^{20}<0\right\}.
\]
The difficulty in achieving (*), in either case, is the existence of complex lines in $\CC$ along
which $P_{2k}$ (in the notation used earlier) is harmonic. Note that the defining functions of
$\OM_1$ and $\OM_2$ are modelled on the polynomial}
\[
G(z) \ = \ |z|^8+\tfrac{15}{7}|z|^2\er(z^6),
\]
that features in the well-known Kohn-Nirenberg example. Recalling the behaviour of $G$, we note that
this {\em rules out the possibility of bumping $\OM_j, \ j=1,2$, by perturbing the higher-order terms} 
in $(z_1,z_2)$.
\smallskip

This last remark suggests that we are committed to perturbing the lowest-order --- or, more
generally, the lowest-weight --- polynomial in $z$ in the defining function of~\eqref{E:localdef}.
This is the first step towards obtaining the possible bumpings of (*). This, in itself, is 
difficult because we are treating
the case where the polynomial $P$ (in the terminology of~\eqref{E:localdef}) is harmonic along 
certain complex subvarieties of $\cplx^{n-1}$. The structure of these exceptional varieties can
be {\em very difficult} to resolve. However, we can handle a large class of polynomials associated
to domains in $\cplx^3$, which includes {\em Example~1} and {\em Example~2}. To obtain the possible 
bumpings, we 
need to construct a non-negative, $\Lam$-homogeneous function $H\in\smoo^{\infty}(\CC)$ such that
\[
P(z_1,z_2)-H(z_1,z_2) \;\; \text{\em is plurisubharmonic $\forall(z_1,z_2)$},
\]
and such that
\[
H(z_1,z_2) \ \geq \ \eps|P(z_1,z_2)| \quad \forall(z_1,z_2)\in\CC
\]
for some $\eps>0$ sufficiently small. That is the focus of this paper. 
The precise results are given in the next section. Since this focused task is
already rather involved, its application to specific function-theoretic estimates will be
tackled in a different article.
\medskip

\section{Statement of results}\label{S:results}

For clarity, we shall {\em initially} present our results in the setting 
of homogeneous plurisubharmonic
polynomials. However, we begin with some notation that is relevant 
to the general setting. Thus, let $P$ be a $(m_1,m_2)$-homogeneous plurisubharmonic
polynomial on $\CC$, and define:
\begin{align}
\LeviDeg{P}  \ &:= \ \{z\in\CC:\hess(P)(z) \ \text{is not strictly positive definite}\},\notag \\
\mathfrak{C}(P) \ &:= \ \text{the set of all irreducible complex curves $V\subset\CC$} \notag \\
&\quad\qquad \text{such that $P$ is harmonic along the smooth part of $V$}, \notag
\end{align}
where $\hess(P)(z)$ denotes the complex Hessian of $P$ at $z\in\CC$. As already mentioned, we
need to tackle the case when $\mathfrak{C}(P)\neq\emptyset$. 
\smallskip

Let us now consider $P$ to be homogeneous of degree $2k$ (plurisubharmonicity ensures that $P$ is of 
even degree). If we assume that 
$\mathfrak{C}(P)\neq\emptyset$, then there is a non-empty collection of complex lines 
through the origin in $\CC$ along which $P$ is harmonic. This follows from the following
observation by Noell:
\smallskip

\begin{result}[Lemma~4.2, \cite{Noell:pfpd93}] Let $P$ be a homogeneous, plurisubharmonic, non-pluriharmonic 
polynomial in $\Cn, \ n\geq 2$. Suppose there exist complex-analytic varieties of positive dimension in $\Cn$
along which $P$ is harmonic. Then, there exist complex lines through the origin in $\Cn$ along which $P$ 
is harmonic.
\end{result} 
\smallskip

\noindent{This collection of complex lines will play a key role. Let us 
denote this non-empty collection of exceptional
complex lines by $\exepc(P)$. What is the structure of $\exepc(P)$~?
An answer to this is provided by the following proposition which is indispensable to
our construction, and which may be of independent interest:}
\smallskip

\begin{proposition}\label{P:harmlines} Let $P$ be a plurisubharmonic, non-pluriharmonic
polynomial in $\CC$ that is homogeneous of degree $2k$. There are at most finitely many complex  
lines passing through $0\in\CC$ along which $P$ is harmonic.
\end{proposition}
\smallskip

\noindent{It is not hard to show that the set of complex lines passing through the origin in $\CC$
along which $P$ is harmonic describes a {\em real}-algebraic subset of $\cplx\mathbb{P}^1$.
Thus, it is possible {\em a priori} that the real dimension of this projective set equals $1$.
The non-trivial part of Proposition \ref{P:harmlines} is that this set is in fact 
zero-dimensional.} 
\smallskip

The interpretation of $\exepc(P)$ for $(m_1,m_2)$-homogeneous polynomials, $m_1\neq m_2$, is 
\begin{align}
\exepc(P) \ :=& \ \text{the class of all curves of the form} \notag \\
&\quad \left\{(z_1,z_2):z_1^{m_1/\gcd(m_1,m_2)}=\zt z_2^{m_2/\gcd(m_1,m_2)}\right\}, \notag \\
& \ \text{$\zt\in\cplx_\infty$, along which $P$ is harmonic}\notag
\end{align}
(with the understanding that $\zt=\infty\Rightarrow$ $P$ is harmonic along
$\{(z_1,z_2)\in\CC:z_2=0\}$). Note how $\exepc(P)$ is just a collection of complex lines
when $m_1=m_2=2k$. As $\mathfrak{C}(P)\neq\emptyset$, perturbing $P$ in
the desired manner becomes extremely messy. However, under certain conditions on 
$\LeviDeg{P}$, we can describe the desired bumping in a relatively brief and precise way.
For this, we need one last definition. An {\em $(m_1,m_2)$-wedge in $\CC$} is defined to be a 
set $\weg$ having the property   
that if $(z_1,z_2)\in\weg$, then $(t^{1/m_1}z_1,t^{1/m_2}z_2)\in\weg \ \forall t>0$. The terms
{\em open $(m_1,m_2)$-wedge} and {\em closed $(m_1,m_2)$-wedge} will have the usual meanings.
Note that when $m_1=m_2=2k$ (the homogeneous case), an $(m_1,m_2)$-wedge is simply a cone.
\smallskip

The problem we wish to solve can be resolved very precisely if the Levi-degeneracy set 
$\LeviDeg{P}$ possesses either one of the following properties:
\begin{enumerate}
\item[{}] {\bf{\em Property~(A):}} $\LeviDeg{P}\setminus\bigcup_{C\in\exepc(P)}C$ contains      
no complex subvarieties of positive dimension and is {\em well separated} from
$\bigcup_{C\in\exepc(P)}C$. In precise terms: there is a closed $(m_1,m_2)$-wedge
$\clweg\subset\CC$ that contains $\LeviDeg{P}\setminus\bigcup_{C\in\exepc(P)}C$ and satisfies
${\rm int}(\clweg)\bigcap(\bigcup_{C\in\exepc(P)}C)=\emptyset$.
\item[{}]
\[
OR
\]
\item[{}] {\bf{\em Property~(B):}} There exists an entire function $\mathcal{H}$ such that
$P$ is harmonic along the smooth part of {\em every} level-curve of $\mathcal{H}$, i.e.
$\LeviDeg{P}=\CC$ and is foliated by these level-curves.
\end{enumerate}
\smallskip
Note that, in some sense, Property~(A) and Property~(B) represent the two extremes of the
complex structure within $\LeviDeg{P}$, given that $\mathfrak{C}(P)\neq\emptyset$. Also
note that polynomials $P$ with the property $\exepc(P)=\emptyset$ are just a special case
of Property~(A).
\smallskip

The reader is referred to Section~\ref{S:intro} for an illustration of these properties. 
The set $\LeviDeg{P}$ for {\em Example~1} (resp.~{\em Example~2}) has Property~(A) 
(Property~(B) resp.). 
We feel that our main results --- both statements and proofs --- are clearest in the 
setting of homogeneous polynomials. Thus, we shall first present our
results in this setting. 
\smallskip

\begin{theorem}\label{T:homobump} Let $P(z_1,z_2)$ be a plurisubharmonic polynomial in $\CC$ 
that is homogeneous of degree $2k$ and has no pluriharmonic terms. Assume that $\LeviDeg{P}$
possesses Property~(A). Define  $\exepc(P):= \ $the set of all complex lines passing through 
$0\in\CC$ along which $P$ is harmonic. Then:

\noindent{1) $\exepc(P)$ consists of finitely many complex lines; and}
\smallskip

\noindent{2) There exist a constant $\delta_0>0$ and a $\smoo^\infty$-smooth function $H\geq 0$ 
that is homogeneous of degree $2k$ 
such that the following hold:}
\begin{enumerate}
\item[(a)] $H^{-1}\{0\}=\bigcup_{L\in\exepc(P)}L$.
\item[(b)] $(P-\delta H)\in\psh(\CC)$ and is strictly plurisubharmonic on 
$\CC\setminus\bigcup_{L\in\exepc(P)}L \ \forall\delta\in(0,\delta_0)$.
\end{enumerate}
\end{theorem}
\smallskip

The next theorem is the analogue of Theorem~\ref{T:homobump} in the case when
$\LeviDeg{P}$ possesses Property~(B) 
\smallskip

\begin{theorem}\label{T:homobumpB} Let $P(z_1,z_2)$ be a plurisubharmonic, non-pluriharmonic
polynomial in $\CC$ that is homogeneous of degree $2k$ and has no pluriharmonic terms. Assume
that $\LeviDeg{P}$ possesses Property~(B). Then: 

\noindent{1) There exist a subharmonic, homogeneous polynomial $U$, and a holomorphic homogeneous
polynomial $F$ such that $P(z_1,z_2)=U(F(z_1,z_2))$; and}
\smallskip
  
\noindent{2) Let $L_1,\dots,L_N$ be the complex lines passing through $0\in\CC$ that constitute
$F^{-1}\{0\}$.
There exists a $\smoo^\infty$-smooth function $H\geq 0$ such that the following hold:}
\begin{enumerate}
\item[(a)] $H^{-1}\{0\}=\bigcup_{j=1}^N L_j$.
\item[(b)]  $(P-\delta H)\in\psh(\CC) \ \forall\delta:0<\delta\leq 1$.
\end{enumerate}
\end{theorem}
\smallskip

Before moving on to the weighted case, let us make a few observations about the proofs of the above 
theorems. Part~(1) of Theorem~\ref{T:homobumpB} follows simply after it is established that $P$ is
constant on the level curves of the function $\mathcal{H}$ occurring in the description of
Property~(B). Part~(2) then follows by constructing a 
bumping of the {\em subharmonic} function $U$. The latter is well understood; the reader is 
referred, for instance, to \cite[Lemma~2.4]{fornaessSibony:cpfwpd89}. 
\smallskip

Proving 
Theorem~\ref{T:homobump} subtler. Essentially, it involves the following steps:
\begin{itemize}
\item {\em Step 1:} By Proposition~\ref{P:harmlines}, $\exepc(P)$ is a finite set, say 
$\{L_1,L_2,\dots,L_N\}$. Let $L_j:=\{(z_1,z_2):z_1=\zt_j z_2 \ \forall z_2\in\cplx\}$ for
some $\zt_j\in\cplx$. We fix a $\zt_j, \ j=1,\dots,N$, and view $P$ in $(\zt,w)$-coordinates 
given by the relations
\[
z_1=(\zt+\zt_j)w, \quad\text{and}\quad z_2=w,
\]
and we define $\jayP(\zt,w):=P((\zt+\zt_j)w,w)$. We expand $\jayP(\zt,w)$ as
a sum of polynomials that are homogeneous in the first variable, with increasing degree in $\zt$.

\item {\em Step 2:} By examining the lowest-degree terms, in the $\zt$-variable, of this 
expansion, one can find cones $\scone{\zt_j}{\sigma_j}$, and functions $H_j$ that are smooth in 
$\scone{\zt_j}{\sigma_j}, \ j=1,\dots,N$, such that $(P-\delta H_j)$ are bumpings of $P$ inside the 
aforementioned cones for each $\delta:0<\delta\leq 1$.

\item {\em Step 3:} Property (A) allows us to patch together all these bumpings $(P-\delta H_j),  
\ j=1,\dots,N$ --- shrinking $\delta>0$ sufficiently when necessary --- to obtain our result.
\end{itemize}
For any $\zt\in\cplx$, the notation $\scone{\zt}{\eps}$, used above, denotes the open cone
\[
\scone{\zt}{\eps} \ := \ \{(z_1,z_2)\in\CC:|z_1-\zt z_2|<\eps|z_2|\}.
\]
Note that $\scone{\zt}{\eps}$ is a conical neighbourhood of the punctured complex line
$\{(z_1=\zt z_2,z_2):z_2\in\cplx\setminus\{0\}\}$.
The details of the above discussion are presented in Sections~\ref{S:tech} and \ref{S:proofsHomo} 
below.     
\smallskip

Continuing with the theme of homogeneous polynomials, we present a result which --- though it
has no bearing on our Main Theorems below --- we found in the course of our investigations 
relating to Theorem~\ref{T:homobumpB}. Since it could be of independent interest, we present
it as:
\smallskip

\begin{theorem}\label{T:2homo}
Let $Q(z_1,z_2)$ be a plurisubharmonic, non-harmonic polynomial
that is homogeneous of degree $2p$ in $z_1$ and $2q$ in $z_2$. Then, $Q$ is of the form
\[
Q(z_1,z_2) \ = \ U(z_1^d z_2^D),
\]
where $d,D\in\zahl_+$ and $U$ is a homogeneous, subharmonic, non-harmonic polynomial.
\end{theorem}
\smallskip

The reader will note the resemblance between the conclusion of the above theorem and Part~(1)
of Theorem~\ref{T:homobumpB}. The proof of Theorem~\ref{T:2homo} is given in 
Section~\ref{S:proofs2homo}.
\smallskip  

The reader will probably intuit that the bumping results for an $(m_1,m_2)$-homogeneous $P$ are
obtained by applying Theorem~\ref{T:homobump} and Theorem~\ref{T:homobumpB} to the pullback of $P$ 
by an appropriate proper holomorphic mapping that homogenises the pullback. We now state our
results in the $(m_1,m_2)$-homogeneous setting. The first of our main 
results --- rephrased for $(m_1,m_2)$-homogeneous polynomials --- is:
\smallskip

\begin{maintheorem}\label{MT:homobump} Let $P(z_1,z_2)$ be an $(m_1,m_2)$-homogeneous, plurisubharmonic
polynomial in $\CC$ that has no pluriharmonic terms. Assume that $\LeviDeg{P}$ possesses
Property~(A). Then:

\noindent{1) $\exepc(P)$ consists of finitely many curves of the form 
\[
\left\{(z_1,z_2):z_1^{m_1/\gcd(m_1,m_2)}=\zt_j z_2^{m_2/\gcd(m_1,m_2)}\right\},
\]
$j=1,\dots,N$, where $\zt_j\in\cplx$; and}
\smallskip
  
\noindent{2) There exist a constant $\delta_0>0$ and a $\smoo^\infty$-smooth $(m_1,m_2)$-homogeneous 
function $G\geq 0$ such that the following hold:}
\begin{enumerate}
\item[(a)] $G^{-1}\{0\}=\bigcup_{C\in\exepc(P)}C$.
\item[(b)] $(P-\delta G)\in\psh(\CC)$ and is strictly plurisubharmonic on
$\CC\setminus\bigcup_{C\in\exepc(P)}C \ \forall\delta\in(0,\delta_0)$.
\end{enumerate}
\end{maintheorem}
\smallskip

The next result tells us what happens when $\LeviDeg{P}$ possesses Property~(B). However, in order to
state this, we will need to refine a definition made in Section~\ref{S:intro}. A real or complex polynomial
$Q$ defined on $\CC$ is said to be {\em $(m_1,m_2)$-homogeneous with weight $r$} if
$Q(t^{1/m_1}z_1,t^{1/m_2}z_2)=t^rQ(z_1,z_2), \ \forall z=(z_1,z_2)\in\CC$ and for every $t>0$. Our second
result can now be stated as follows 
\smallskip

\begin{maintheorem}\label{MT:homobumpB} Let $P(z_1,z_2)$ be an $(m_1,m_2)$-homogeneous, plurisubharmonic
polynomial in $\CC$ that has no pluriharmonic terms. Assume that $\LeviDeg{P}$ possesses
Property~(B). Then:

\noindent{1) There exist a holomorphic polynomial $F$ that is $(m_1,m_2)$-homogeneous with weight
$1/2\nu$, $\nu\in\zahl_+$, and a subharmonic, polynomial $U$ that is homogeneous of degree $2\nu$
such that $P(z_1,z_2)=U(F(z_1,z_2))$; and}
\smallskip

\noindent{2) There exists a $\smoo^\infty$-smooth $(m_1,m_2)$-homogeneous
function $G\geq 0$ such that the following hold:}
\begin{enumerate}
\item[(a)] $G^{-1}\{0\}=\bigcup_{C\in\exepc(P)}C$.
\item[(b)] $(P-\delta G)\in\psh(\CC) \ \forall\delta:0<\delta\leq 1$.
\end{enumerate}
\end{maintheorem}
\medskip
    
\section{Some technical propositions}\label{S:tech}

The goal of this section is to state and prove several results of a technical nature that will
be needed in the proof of Theorem~\ref{T:homobump}. Key among these is Proposition~\ref{P:harmlines},
stated above. We begin with its proof.
\smallskip

\begin{custom}\label{Pr:harmlines}
\begin{proof}[{\bf The proof of Proposition \ref{P:harmlines}}]
Assume that $P$ has at least one complex line, say $L$, passing through $0\in\CC$ such that
$P|_L$ is harmonic. Since $P$ non-pluriharmonic, there exists a complex line $\cpln\neq L$ passing
through $0$ such that $P|_\cpln$ is subharmonic and non-harmonic. By making a complex-linear change of
coordinate if necessary, let us work in global holomorphic coordinates $(z,w)$ with respect to which
\[
L \ = \ \{(z,w)\in\CC \ | \ z=0\},\quad\qquad
\cpln \ = \ \{(z,w)\in\CC \ | \ w=0\}.
\]
Let $M$ be the lowest degree to which $z$ and $\zbar$ occur among the monomials constituting
$P$. Let us write
\[
P(z,w) \ = \ \sum_{j=M}^{2k}\sum_{\alpha+\beta=j \atop \mu+\nu=2k-j}
                C^j_{\alpha\beta\mu\nu}z^\alpha\zbar^\beta w^\mu\wbar^\nu.
\]
Notice that by construction
\[
s(z) \ := \ \sum_{\alpha+\beta=2k}C^{2k}_{\alpha\beta 00}z^\alpha\zbar^\beta \quad\text{\em is
subharmonic and non-harmonic.}
\]
We shall make use of this fact soon. We now study the restriction of $P$ along the
complex lines $L^\zt := \{(z=\zt w,w):w\in\cplx\}$. Note that
\begin{equation}\label{E:subharm}
P(\zt w,w) \ = \ \sum_{m+n=2k}\left\{\sum_{j=M}^{2k}\sum_{\alpha+\beta=j}
                C^j_{\alpha\beta,(m-\alpha),(n-\beta)}\zt^\alpha\ztbar^\beta\right\}
                w^m\wbar^n.
\end{equation}
Denoting the function $w\mapsto P(\zt w,w)$ by $P_\zt(w)$, let us use the notation
\begin{equation}
\laplc P_\zt(w) \ \equiv \ \sum_{m+n=2k}\phi_{mn}(\zt)w^{m-1}\wbar^{n-1}. \notag
\end{equation}
Note that
\begin{equation}\label{E:zero/harm}
\{\zt\in\cplx:P|_{L^\zt} \ \text{is harmonic}\} \ = \
        \{\zt\in\cplx:\phi_{mn}(\zt)=0 \ \forall m,n\geq 0 \ \backepsilon \ m+n=2k\}.
\end{equation}
Since $P_\zt$ is subharmonic $\forall\zt\in\cplx$, the coefficient of the $|w|^{2k}$ term
occurring in \eqref{E:subharm} --- i.e. the polynomial $\phi_{kk}(\zt)/k^2$ --- is
non-negative, and must be positive at $\zt\in\cplx$
whenever $P_\zt$ is non-harmonic. To see this, assume for the moment that, for some
$\zt^*\in\cplx$, $P_{\zt^*}$ is non-harmonic but $\phi_{kk}(\zt^*)\leq 0$. Then,
as $P_{\zt^*}(w)$ is real-analytic, $\laplc P_{\zt^*}(e^{i\tht})>0$ {\em except} at
finitely many values of $\tht\in[0,2\pi)$. Hence, we have
\[
0 \ < \ \int_0^{2\pi}\laplc P_{\zt^*}(e^{i\tht}) \ d\tht \ = \
\int_0^{2\pi}\left\{\sum_{m+n=2k}\phi_{mn}(\zt^*)e^{i(m-n)\tht}\right\} d\tht \
= \ 2\pi \ \phi_{kk}(\zt^*).
\]
The assumption that $\phi_{kk}(\zt^*)\leq 0$ produces a contradiction in the above
inequality, whence our assertion. Thus $\phi_{kk}\geq 0$.
\smallskip

Let us study the zero-set of $\phi_{kk}$. We first consider the highest-order terms of
$\phi_{kk}$, namely
\[
\sum_{\alpha+\beta=2k}k^2C^{2k}_{\alpha\beta,(k-\alpha),(k-\beta)}\zt^\alpha\ztbar^\beta \ =
\ k^2C^{2k}_{kk00}|\zt|^{2k}.
\]
The tidy reduction on the right-hand side occurs because we must only consider those
pairs of subscripts $(\alpha,\beta)$ such that $(k-\alpha)\geq 0$ and $(k-\beta)\geq 0$.
Note that $C^{2k}_{kk00}$ is the coefficient of the $|z|^{2k}$ term of $s(z)$,
which is subharmonic and non-harmonic. Thus $C^{2k}_{kk00}>0$. The nature of the
highest-order term of $\phi_{kk}$ shows that $\phi_{kk}\neq 0$ when $|\zt|$ is sufficiently large.
We have thus inferred the following

\noindent{{\bf Fact:} {\em $\phi_{kk}$ is a real-analytic function such that
$\phi_{kk}\geq 0$, $\phi_{kk}\not\equiv 0$, and such that $\phi_{kk}^{-1}\{0\}$ is
compact.}}
\smallskip

Since $\phi_{kk}$ is a real-analytic function on $\cplx$ that is not identically zero,
$\rDim[\phi_{kk}^{-1}\{0\}]\leq 1$. Let us assume that $\rDim[\phi_{kk}^{-1}\{0\}]=1$.
We make the following

\noindent{{\bf Claim:} {\em The function $\phi_{kk}$ is subharmonic}}

\noindent{Consider the function
\[
S(\zt) \ := \ \frac{k^2}{2\pi}\int_0^{2\pi}P(\zt e^{i\tht},e^{i\tht}) \ d\tht \ .
\]
Notice that
\begin{equation}\label{E:altformS}
S(\zt) \ := \ \frac{k^2}{2\pi}\int_0^{2\pi}\left\{\sum_{m+n=2k}
        \frac{\phi_{mn}(\zt)}{mn}e^{i(m-n)\tht}\right\} d\tht \
= \ \phi_ {kk}(\zt).
\end{equation}
Furthermore, denoting the function $\zt\mapsto P(\zt w,w)$ by $P_w(\zt)$, we see that
\begin{equation}\label{E:laplacianS}
\mixderiv{\zt}{\zt}{S}(\zt) \ = \  
\frac{k^2}{2\pi}\int_0^{2\pi}\mixderiv{\zt}{\zt}{P_{e^{i\tht}}}(\zt) \ d\tht \  = \
\frac{k^2}{2\pi}\int_0^{2\pi}\mixderiv{z}{z}{P}(\zt e^{i\tht},e^{i\tht}) \ d\tht \
\geq \ 0.
\end{equation}
The last inequality follows from the plurisubharmonicity of $P$. By
\eqref{E:altformS} and \eqref{E:laplacianS}, the above claim is established.}
\smallskip

By assumption, $\phi_{kk}^{-1}\{0\}$ is a $1$-dimensional real-analytic variety, and
we have shown that it is compact. Owing to compactness, there is an open, connected
region $D\Subset\cplx$ such that $\bdy D$ is a piecewise real-analytic
curve (or a disjoint union of piecewise real-analytic curves). Furthermore
\[
\phi_{kk}(\zt) \ > \ 0 \ \forall\zt\in D, \quad\qquad \phi_{kk}(\zt) \ = \ 0 \
\forall\zt\in\bdy D.
\]
But the above statement contradicts the Maximum Principle for $\phi_{kk}$, which is
subharmonic. Hence, our assumption must be wrong. Thus, $\phi^{-1}\{0\}$ is a discrete
set; and being compact, it is a finite set. In view of \eqref{E:zero/harm}, we
have $\{\zt\in\cplx:P|_{L^\zt} \ \text{is harmonic}\}\subseteq\phi_{kk}^{-1}\{0\}$,
which establishes our result.
\end{proof}   
\end{custom}
\smallskip

Our next result expands upon of the ideas summarised in {\em Step~1} and {\em Step~2} in
Section~\ref{S:results} above. But first, we remind the reader that, given a function $G$ 
of class $\smoo^2$ in an open set $U\subset\CC$ and a vector $v=(v_1,v_2)\in\CC$, the 
Levi-form $\levi{G}(z;v)$ is defined as
\[
\levi{G}(z;v):=\sum_{j,k=1}^2\bdy^{2}_{j\overline{k}}G(z)v_j\overline{v_k}.
\]
A comment about the hypothesis imposed on $P$ in the following result: the $P$ below is the
prototype for the polynomials $\widetilde{{}^jP}$ discussed in {\em Step~1} of
Section~\ref{S:results}. 
\smallskip

\begin{proposition}\label{P:conebump} Let $P(z_1,z_2)$ be a plurisubharmonic polynomial that is
homogeneous of degree $2k$, and contains no pluriharmonic terms. Assume that $P(z_1,z_2)$
vanishes identically along $L:=\{(z_1,z_2):z_1=0\}$ and that there exists an $\eps>0$ such that
$P$ is strictly plurisubharmonic in the cone $\left(\scone{0}{\eps}\setminus L\right)$.
There exist constants $C_1$ and
$\sigma>0$ --- both of which depend only on $P$ --- and a non-negative function $H\in\smoo^\infty(\CC)$ 
that is homogeneous of degree $2k$ such that:
\begin{enumerate}
\item[a)] $H(z_1,z_2)>0$ when $0<|z_1|<\sigma|z_2|$.
\item[b)] For any $\delta:0<\delta\leq 1$:
\begin{align}
\levi{(P-\delta H)}(z;(V_1,V_2)) \ &\geq \ C_1|z_1|^{2(k-1)}\left(|V_1|^2
	\left(1-\frac{\delta}{2}\right)+\left|\frac{z_1}{z_2}V_2\right|^2\right)\notag \\
        &\qquad \; \forall z: 0\leq |z_1|<\sigma|z_2|, \forall V\in\CC.\notag
\end{align}
\end{enumerate}
\end{proposition}
\begin{proof} Let us write
\[
P(z_1,z_2) \ = \ \sum_{j=M}^{2k}Q_j(z_1,z_2),
\]
where each $Q_j$ is the sum of all monomials of $P$ that involve powers of $z_1$ and
$\zbar_1$ having total degree $j, \ M\leq j\leq 2k$.
We make a Levi-form calculation. Let us define $\phi:=\rg{w}$ and
$\alpha:=\rg{\zt}$. With this notation, the Levi-form of $P$ at points $(z_1,z_2)=(\zt w,w)$ can
be written as
\begin{multline}\label{E:leviform}
\levi{P}((\zt w,w); V) \\
 := \ |w|^{2(k-1)}\sum_{j=M}^{2k}|\zt|^{j-2}\times(V_1 \quad \zt V_2) \   
                                \begin{pmatrix}
                                \ T_{11}^{(j)}(\phi,\alpha) & T_{12}^{(j)}(\phi,\alpha) \ \\
                                {} & {} \\
                                \ \overline{T_{12}^{(j)}(\phi,\alpha)} & T_{22}^{(j)}(\phi,\alpha) \
                                \end{pmatrix} \ \begin{pmatrix}
                                                \ \overline{V_1} \ \\
                                                {} \\
                                                \ \overline{\zt V_2} \
                                                \end{pmatrix} \ .
\end{multline}
where $T_{11}^{(j)}$, $T_{12}^{(j)}$ and $T_{22}^{(j)}$ are trigonometric polynomials obtained when
$\levi{Q_j}((\zt w,w);V)$ is written out using the substitutions
\[
w=|w|e^{i\phi} \qquad\text{and}\qquad \zt \ = \ |\zt|e^{i\alpha},
\]
$j=M,\dots,2k$. Now, consider the matrix-valued functions 
$\funz \ \bcdot):\tor\times\tor\lrarw\cplx^{2\times 2}$ defined by
\[
\funz\tht_1,\tht_2) \ := \ \sum_{j=M}^{2k}r^{j-2}\begin{pmatrix}
                                \ T_{11}^{(j)}(\tht_1,\tht_2) & T_{12}^{(j)}(\tht_1,\tht_2) \ \\
                                {} & {} \\
                                \ \overline{T_{12}^{(j)}(\tht_1,\tht_2)} & T_{22}^{(j)}(\tht_1,\tht_2) \
                                \end{pmatrix}.
\]
Here $\tor$ stands for the circle, and $\funz \ \bcdot)$ is a periodic function in $(\tht_1,\tht_2)$.
Let us introduce the notation $\ev(M):=\min_{\lambda\in\sigma(M)}|\lambda|$ --- i.e. the modulus of
the least-magnitude eigenvalue of the matrix $M$. Since $\funz \ \bcdot)$ takes values in the class of
{\em $2\times 2$ Hermitian matrices}, we get
\begin{equation}\label{E:eigenRel}
\ev\left[\funz\tht_1,\tht_2)\right] \ = \ \frac{|{\rm det}(\funz\tht_1,\tht_2))|}
								{\|\funz\tht_1,\tht_2)\|_2},
\end{equation}
where the denominator represents the operator norm of the matrix $\funz \ \bcdot)$.
Since $P\in{\sf spsh}\left(\scone{0}{\eps}\setminus L\right)$, comparing $\funz \ \bcdot)$ with the 
Levi-form computation \eqref{E:leviform}, we see that, provided $r\in(0,\eps)$
\begin{itemize}
\item $\funz \ \bcdot)$ takes strictly positive-definite values on $\tor\times\tor$;
\item in view of the above and by the relation \eqref{E:eigenRel}, 
$\ev\left[\funz \ \bcdot)\right]$ are continuous functions on $\tor\times\tor$; and
\item owing to the preceding two facts, an estimate of
$\ev\left[\funz \ \bcdot)\right]$ using the quadratic formula tells us that
there exists a $C_1>0$ such that 
\begin{align}
\ev\left[\funz\tht_1,\tht_2)\right] \ \geq \ C_1r^{2(k-1)}\quad 
						&\forall(\tht_1,\tht_2)\in\tor\times\tor \notag\\
&\text{(provided $r:0<r<\eps$).}\notag
\end{align}
\end{itemize}
Substituting $\tht_1=\phi$ and $\tht_2=\alpha$ above therefore gives us
\begin{align}
\levi{P}((\zt w,w);V) \ \geq \ C_1|\zt|^{2(k-1)}|w|^{2(k-1)}\left(|V_1|^2+|\zt V_2|^2\right)
        \quad &\forall\zt:0\leq|\zt|<\eps, \notag \\
	&\forall w\in\cplx \ \text{and} \ \forall V\in\CC.\notag
\end{align}

Let us now define 
\[
H(z_1,z_2) \ := \ \frac{C_1}{2k^2}|z_1|^{2k},
\]
and fix a $\sigma$ such that $0<\sigma<\eps$. Then
\begin{align}
\levi{(P-\delta H)}((\zt w,w);V) \ &\geq \ C_1|\zt|^{2(k-1)}|w|^{2(k-1)}\left(|V_1|^2
        \left(1-\frac{\delta}{2}\right)+|\zt V_2|^2\right) \notag\\
        &\qquad \; \forall\zt: 0\leq |\zt|<\sigma, \ \forall w\in\cplx \ \text{and} \ 
		\forall V\in\CC.\notag
\end{align}
The last inequality is the desired result.
\end{proof}
\smallskip
                                
Our next lemma is the first result of this section that refers to $\Lambda$-homogeneous polynomials
for a general ordered pair $\Lambda$. This result will provide a useful first step towards tackling
the theorems in Section~\ref{S:results} pertaining to plurisubharmonic polynomials that possess 
Property~(B).
\smallskip

\begin{lemma}\label{L:homoImp} Let $P$ be a $(m_1,m_2)$-homogeneous plurisubharmonic, non-pluriharmonic
polynomial in $\CC$ having Property~(B). Then, there exists a rational number $q^*$
and a complex polynomial $F$ that is $(m_1,m_2)$-homogeneous with weight $q^*$ such that $P$ is harmonic along
the smooth part of the level sets of $F$.
\end{lemma}
\begin{proof} 
Let us first begin by defining $M:=$ the largest positive integer $\mu$ such that there exists some
$f\in\hol(\CC)$ and $f^\mu=\mathcal{H}$ (here, $\mathcal{H}$ is as given by Property~(B)). Define 
$F$ by the relation $F^M=\mathcal{H}$.
Observe that the hypotheses of this lemma continue to hold when $\mathcal{H}$ is replaced by $F$.
\smallskip

Let $\dil_t$ denote the dilations
$\dil_t:(z_1,z_2)\longmapsto(t^{1/m_1}z_1,t^{1/m_2}z_2)$, and define the set
$\mathfrak{S}(P):=\{z\in\CC:\hess(P)(z)=\zerov\}$. Since, by hypothesis, 
${\rm det}\left[\hess(P)\right]\equiv 0$, we have
\[
\mathfrak{S}(P) \ = \ \{z\in\CC:(\Mixdrv{1}{1}P+\Mixdrv{2}{2}P)(z)
			={\rm tr}(\hess(P))(z)=0\}.
\]
As $P$ is not pluriharmonic, ${\rm dim}_{\RR}\left[\mathfrak{S}(P)\right]\leq 3$, and
$\mathfrak{S}(P)$ is a closed subset of $\CC$. Hence, by the open-mapping theorem
\[
W(P) \ := \ F(\CC\setminus\mathfrak{S}(P))\setminus\{0\}
\]
is a non-empty open subset of $\cplx$. Pick any $c\in W(P)$ and set
$V_c:=F^{-1}\{c\}$. Then $V_c\bigcap\mathfrak{S}(P)=\emptyset$; and, in view of the
transformation law for the Levi-form and the fact that $P$ is $(m_1,m_2)$-homogeneous,
$\dil_t(V_c)\bigcap\mathfrak{S}(P)=\emptyset \ \forall t>0$. We now make the following
\smallskip

\noindent{{\bf Claim.} {\em Each $\dil_t(V_c), \ t>0$, is contained in some level set of $F$
($c\in W(P)$ as assumed above).}}

\noindent{To see this we note that by the transformation law for the Levi-form, $P$ is harmonic
along the smooth part of $\dil_t(V_c) \ \forall t>0$. If $\dil_t(V_c)$, for some $t>0$, is not contained
in any level set, then there would exist a non-empty, Zariski-open subset, say $S$, of the curve 
$\dil_t(V_c)$ such that:
\begin{itemize}
\item  for each $\zt\in S$, there is some level set of $F$ passing through $\zt$
that is transverse to $\dil_t(V_c)$ at $\zt$; and
\item owing to the above, $\levi{P}(\zt; \ \bcdot)\equiv 0 \ \forall\zt\in S$.
\end{itemize}
But the second statement above cannot be true because $\dil_t(V_c)\bigcap\mathfrak{S}(P)=\emptyset$.
Hence the claim.}
\smallskip

We now define the following function $\phi:W(P)\times(0,\infty)\lrarw W(P)$ defined by
\[
\phi(c,t) \ := \ \text{the number $b\in\cplx$ such that $F^{-1}\{bc\}\supset\dil_t(V_c)$}.
\]
This $\phi$ is well-defined in view of the Claim above. Define the following two sets
\begin{align}
{\sf Supp}(F) \ &:= \ \left\{\alpha\in\Nat^2:\frac{\partial^{|\alpha|} F}{\partial z^\alpha}(0)\neq 0\right\},
\notag \\
\indx(F) \ &:= \ \left\{\frac{\alpha_1 m_2+\alpha_2 m_1}{m_1 m_2}:(\alpha_1,\alpha_2)\in{\sf Supp}(F)\right\}.
\notag
\end{align}
By construction of $\indx(F)$, we can write
\[
F(z) = \sum_{q\in\indx(F)}P_q(z_1,z_2)
\]
where each $P_q$ is a regrouping of the terms in the Taylor expansion of $F$ around $z=0$ such that
$P_q$ is $(m_1,m_2)$-homogeneous with weight $q$. Pick a $c^0\in W(P)$ such that $V_{c^0}$ is a
nonsingular curve {\em and} $\bigcup_{t>0}\dil_t(V_{c^0})\varsupsetneq V_{c^0}$ (we can do this because a
$V_c$ having these properties is generic as $c$ varies through $W(P)$) and
let $z_0$ be such that $F(z^0)=c^0$. Then:
\begin{align}\label{E:homolevels1}
\sum_{q\in\indx(F)}P_q(t^{1/m_1}z_1,t^{1/m_2}z_2)\
= \ \phi(F(z^0),t)\sum_{q\in\indx(F)}P_q&(z_1,z_2) \\
 \forall&(z_1,z_2)\in V_{c^0} \ \text{and $\forall t>0$.} \notag
\end{align}
The above equation holds for all $z\in V_{c^0}$ due to our above Claim. On the other hand:
\[
\sum_{q\in\indx(F)}P_q(t^{1/m_1}z_1,t^{1/m_2}z_2) \ = \
\sum_{q\in\indx(F)}t^qP_q(z_1,z_2) \ \forall (z_1,z_2)\in V_{c^0} \ \text{and $\forall t>0$.}
\]
This, along with \eqref{E:homolevels1}, gives us
\begin{equation}\label{E:homolevels2}
\sum_{q\in\indx(F)}(\phi(F(z^0),t)-t^q)\left.P_q\right|_{V_{c^0}} \ \equiv \ 0 \;\; \forall t>0.
\end{equation}
Since $V_{c^0}$ is nonsingular, zero is a unique linear combination of the $\left.P_q\right|_{V_{c^0}}$'s
(see, for instance, \S{II.2.1},~Theorem~4, of Shafarevich's
\cite{shafarevich:bagI94}). I.e. we can ``compare coefficients'' in \eqref{E:homolevels2}
to get:
\[
\left.P_q\right|_{V_{c^0}}\not\equiv 0 \ \Longrightarrow \
\phi(F(z^0),t)=t^q \ \forall t>0 \ \text{and $\forall z\in \CC\setminus\mathfrak{S}(P)$.}
\]
Since this holds true as $t$ varies through $\RR_+$, we conclude that there exists a
$q^*\in\indx(F)$ such that $\left.P_q\right|_{V_{c^0}}\equiv 0$ if $q\neq q^*$. Now, consider
the set
$\hyprs:=\bigcup_{t>0}\dil_t(V_{c^0})$. By our choice of $c^0$, $\hyprs$ is a real hypersurface.
Owing to the homogeneity of the $P_q$'s, we have
\[
F|_{\hyprs} \ = \ \sum_{q\in\indx(F)}\left.P_q\right|_{\hyprs} \ = \left.P_{q^*}\right|_{\hyprs}.
\]
But since $F$ and $P_{q^*}$ agree on a real hypersurface, $F\equiv P_{q^*}$. Hence the
result.
\end{proof}
\smallskip

Our next result will play a key role in the proofs pertaining to polynomials having Property~(B).
It is a rephrasing of \cite[Lemma~2.4]{fornaessSibony:cpfwpd89}. Since it is an {\em almost} 
direct rephrasing, we shall not prove this result. Readers are, however, referred to the remark
immediately following this lemma.
\smallskip

\begin{lemma}\label{L:subhBump} Let $U:\cplx\to\RR$ be a real-analytic,
subharmonic, non-harmonic function that is homogeneous of degree $j$. There exist a positive 
constant $C\equiv C(U)$ --- i.e. depending only on $U$ --- and a $2\pi$-periodic
function $h\in\smoo^\infty(\RR)$ such that:
\begin{enumerate}
\item[a)] $0<h(x)\leq 1 \ \forall x\in\RR$.
\item[b)] $\laplc\left(U-\delta|\bcdot|^j h\circ\rg{\bcdot}\right)(z)
\geq \delta C|z|^{j-2} \ \forall z\in\cplx$ 
and $\forall\delta:0<\delta\leq 1$.
(Here $\rg{\bcdot}$ refers to any continuous branch of the argument.)
\end{enumerate}
\end{lemma}

\begin{remark}
The $\delta>0$ appearing in the above lemma must not be confused for the $\delta$ appearing
in the statement \cite[Lemma~2.4]{fornaessSibony:cpfwpd89}. The latter $\delta$ is a universal
constant which is a component of the constant $C(U)$ in our notation. If we denote the
$\delta$ of \cite[Lemma~2.4]{fornaessSibony:cpfwpd89} by $\delta_{{\rm univ}}$, then our
$C(U)$ is a polynomial function of $\delta_{{\rm univ}}$ and
\[
\text{(in the notation of \cite{fornaessSibony:cpfwpd89})} \;\; \|U\| \ := \ 
\sup_{|z|=1}|U(z)|.
\]
\end{remark}
\medskip

\section{Proofs of the theorems in the homogeneous case}\label{S:proofsHomo}

We begin with the proof of Theorem~\ref{T:homobump}, most of whose ingredients are now available
from the previous section. However, we need one last result, which is derived from 
\cite{Noell:pfpd93}.
\smallskip

\begin{result}[Version of Prop.~4.1 in \cite{Noell:pfpd93}]\label{R:noell}
Let $P$ be a plurisubharmonic polynomial on $\Cn$ that is homogeneous of degree $2k$. Let
$\omega_0$ be a connected component of $\LeviDeg{P}\setminus\{0\}$ having the following
two properties:
\begin{itemize}
\item[a)] There exist closed cones $\clcone_1$ and $\clcone_2$ such that
\[
\omega_0 \subset \ {\rm int}(\clcone_1) \subset \ \clcone_1\setminus\{0\}  
\varsubsetneq \ {\rm int}(\clcone_2),
\]
and such that $\clcone_2\bigcap(\LeviDeg{P}\setminus\overline{\omega_0})=\emptyset$.
\item[b)] $\omega_0$ does not contain any complex-analytic subvarieties of positive dimension
along which $P$ is harmonic.
\end{itemize}
Then, there exist a smooth function $H\geq 0$ that is homogeneous of degree $2k$ and constants
$C,\eps_0>0$, which depend only on $P$, such that  ${\rm int}(\clcone_2)=\{H>0\}$ and such that 
for each $\eps:0<\eps\leq\eps_0$, $\levi{(P-\eps H)}(z;v)\geq C\eps\|z\|^{2(k-1)} \ 
\|v\|^2 \ \forall (z,v)\in\clcone_2\times\Cn$.  
\end{result}
\smallskip

\noindent{The above result is not phrased in precisely these terms in
\cite[Proposition~4.1]{Noell:pfpd93}. The proof of the latter proposition was derived from
a construction pioneered by Diederich and Fornaess in \cite{diederichFornaess:pd:bspef77}.
A close comparison of the proof of \cite[Proposition~4.1]{Noell:pfpd93} with the 
Diederich-Fornaess construction reveals that incorporating the assumption (a) in 
Result~\ref{R:noell} allows us to obtain the above ``localised'' 
version of \cite[Proposition~4.1]{Noell:pfpd93}.
\smallskip
 
\begin{custom}\begin{proof}[{\bf The proof of Theorem~\ref{T:homobump}.}] 
Note that Part~(1) follows simply from Proposition~\ref{P:harmlines}. Hence, let us 
denote the set $\exepc(P)$ by $\exepc(P)=\{L_1,\dots,L_N\}$.
Let $\clcone$ be the closed cone whose existence is guaranteed by Property~(A). By assumption,
we can find a slightly larger cone $\clBcone$ such that
\[
\LeviDeg{P}\setminus(\bigcup_{j=1}^N L_j) \subset \ {\rm int}(\clcone) 
\subset \ \clcone\setminus\{0\} \varsubsetneq \ {\rm int}(\clBcone),
\]
and such that $\clBcone\bigcap(\bigcup_{j=1}^N L_j)=\{0\}$. Hence, in view of
Result~\ref{R:noell}, we can find a smooth function $H_0\geq 0$ that is homogeneous
of degree $2k$, and constants $C_0,\eps_0>0$ such that
\begin{align}
\{z:H_0>0\} \ &= \ {\rm int}(\clBcone), \notag \\
\levi{(P-\delta H_0)}&(z;v) \label{E:leviNoell} \\
&\geq \ C_0\delta\|z\|^{2(k-1)} \ \|v\|^2 \quad 
\forall (z,v)\in\clBcone\times\CC, \ \text{and} \ \forall\delta\in(0,\eps_0).\notag
\end{align}
Without loss of generality, we may assume that each $L_j$ is of the form
$L_j=\{(\zt_jz_2,z_2):z_2\in\cplx\}$. Applying Proposition~\ref{P:conebump} to
\[
\jayP(z_1,z_2) \ := \ P(z_1+\zt_j z_2,z_2)
\]
we can find:
\begin{itemize}
\item a constant $B_1>0$ that depends only on $P$;
\item constants $\sigma_j>0, \ j=1,\dots,N$, that depend only on $P$ and $j$; and
\item functions $H_j\in\smoo^\infty(\CC)$ that are homogeneous of degree $2k$;
\end{itemize}
such that 
\begin{align}\label{E:leviBS}
\levi{(P-\delta H_j)}&(z;v) \notag \\
&\geq \ \delta B_1|z_1-\zt_j z_2|^{2(k-1)}
	\left(|v_1-\zt_jv_2|^2\left(1-\frac{\delta}{2}\right) 
	+ \left|\frac{z_1-\zt_jz_2}{z_2}v_2\right|^2\right)\notag \\
        & \ \forall (z,v)\in\left[\smcone{\zt_j}{\sigma_j}\setminus\{0\}\right]\times\CC \ 
	\text{and} \ \forall\delta:0<\delta\leq 1.\notag
\end{align}
The reader is reminded that $\scone{\zt_j}{\sigma_j}$ denotes an open cone, as
introduced in Section~\ref{S:results}, and that the right-hand side above is finite.
Let $\sig{j}>0, \ j=1,\dots,N$, be so small that
\begin{align}
2\sig{j} \ &\leq \ \sigma_j, \quad j=1,\dots,N,\notag\\ 
(\smcone{\zt_j}{2\sig{j}}\cap S^3)\bigcap(\clcone_*\cap S^3) \ &= \ \emptyset \quad\forall j\leq N,\notag\\
(\smcone{\zt_j}{2\sig{j}}\cap S^3)\bigcap(\smcone{\zt_k}{2\sig{k}}\cap S^3) \ &= \ \emptyset
\quad\text{if $j\neq k$}.
\end{align}
Here, $S^3$ denotes the unit sphere in $\CC$. We introduce these new parameters in order
to patch together all the above ``localised'' bumpings.
\smallskip

Let us now define
\[
V_j \ := \ \smcone{\zt_j}{\sig{j}}\cap S^3, \quad\text{and}\quad
	U_j \ := \ \smcone{\zt_j}{2\sig{j}}\cap S^3.
\]
Let $\chi_j:S^3\lrarw [0,1]$ be a smooth cut-off function such that 
$\chi_j|_{V_j}\equiv 1$ and ${\rm supp}(\chi_j)\subset U_j, \ j=1,\dots,N$. Let us
define $\Psi_j(z):=\chi_j(z/\|z\|) \ \forall z\in\CC\setminus\{0\}$. Finally, let us use
the expression $\Psi_jH_j$ to denote the homogeneous function defined as
\[
\Psi_jH_j(z) \ := \ \begin{cases}
			\Psi_j(z)H_j(z), &\text{if $z\neq 0$,} \\
			0, &\text{if $z=0$.}
			\end{cases}
\]
Note that $\Psi_jH_j\in\smoo^\infty(\CC)$ and is homogeneous of degree $2k$. Let us now 
estimate the Levi-form of $(P-\delta\Psi_jH_j)$ on 
$\scone{\zt_j}{2\sig{j}}\setminus\smcone{\zt_j}{\sig{j}}$. Since, by construction,
$(P-\delta\Psi_jH_j)$ is strictly plurisubharmonic on 
$\scone{\zt_j}{2\sig{j}}\setminus\smcone{\zt_j}{\sig{j}}$, and strict plurisubharmonicity is
an open condition, we infer from continuity and homogeneity:
\begin{multline}\label{E:fact}
\exists\eps\ll 1 \ \text{\em such that $(P-\delta\Psi_jH_j)$ is strictly} \\
	\text{\em plurisubharmonic on $\scone{\zt_j}{\sig{j}+\eps} \ \forall j=1,\dots,N$, and
		for each $\delta:0<\delta\leq 1/2$.}
\end{multline}
For the moment, let us {\em fix} $j\leq N$. Note that, by construction, we can find a
$\beta_j>0$ such that 
\[
(1-\Psi_j)(z) \ \geq \ \beta_j \quad\forall z\in
		\scone{\zt_j}{2\sig{2}}\setminus\smcone{\zt_j}{\sig{j}+\eps}.
\]
Furthermore, since $(\sig{j}+\eps)|z_2|<|z_1-\zt_j z_2|<2\sig{j}|z_2|$ in the
cone $\scone{\zt_j}{2\sig{2}}\setminus\smcone{\zt_j}{\sig{j}+\eps}$, applying this to
\eqref{E:leviBS} gives us small constants $\gamma_j, c_j>0$ such that
\begin{align}
\levi{(P-\delta H_j)}(z;v) \ \geq \ \gamma_j\|z\|^{2(k-1)}\|v\|^2
\quad &\forall z\in\scone{\zt_j}{2\sig{2}}\setminus\smcone{\zt_j}{\sig{j}+\eps}, 
	\label{E:leviest1}\\
	&\forall v\in\CC,  \ \text{and} \ \forall\delta: 0<\delta\leq 1/2;\notag\\
\levi{P}(z;v) \ \geq \ c_j\|z\|^{2(k-1)}\|v\|^2
\quad &\forall (z,v)
	\in(\scone{\zt_j}{2\sig{2}}\setminus\smcone{\zt_j}{\sig{j}+\eps})\times\CC.
	\label{E:leviest2}
\end{align}
From the estimates \eqref{E:leviest1} and \eqref{E:leviest2}, we get
\begin{align}
\levi{(P-\delta\Psi_jH_j)}(z;v) \ &= \ \Psi_j(z)\levi{(P-\delta H_j)}(z;v)+
	(1-\Psi_j)(z)\levi{P}(z;V) \notag\\
	&\qquad -\delta H_j(z)\levi{\Psi_j}(z;v)
		-2\delta\er\left[\sum_{\mu,\nu=1}^2
		\monodrv{\mu}\Psi_j(z)\monodrv{\overline{\nu}}\Psi_k(z)
		v_\mu\overline{v_\nu}\right] \notag\\
&\geq \ \|z\|^{2(k-1)}\|v\|^2\left(\gamma_j\Psi_j(z)+c_j\beta_j\right) \notag \\
&\quad \ -\delta\left(2\left|\sum_{\mu,\nu=1}^2
		\monodrv{\mu}\Psi_j(z)\monodrv{\overline{\nu}}\Psi_k(z)
		v_\mu\overline{v_\nu}\right| + |H_j(z)\levi{\Psi_j}(z;v)|\right) \notag\\
&\qquad\qquad \forall z\in\scone{\zt_j}{2\sig{2}}\setminus\smcone{\zt_j}{\sig{j}+\eps} \
		\text{and} \ \forall v\in\CC. \notag
\end{align}
Finally, we can find constants $K_1, K_2, K_3>0$ and a $\delta_j>0$ that is so small
that, in view of the above calculation, we can make the following estimates:
\begin{align}\label{E:leviPos}
\levi{(P-\delta\Psi_jH_j)}(z;v) \ &\geq \ 
	\|z\|^{2(k-1)}\|v\|^2\left(c_j\beta_j-2\delta K_1\right)-2\delta K_2|H_j(z)|\|v\|^2 \\
	& \geq \ \frac{c_j\beta_j}{2}\|z\|^{2(k-1)}\|v\|^2 \notag \\
	& \geq \ \delta K_3\|z\|^{2(k-1)}\|v\|^2 \notag \\ 
	&\qquad\forall z\in\scone{\zt_j}{2\sig{2}}\setminus\smcone{\zt_j}{\sig{j}+\eps},\notag\\
        &\qquad\forall v\in\CC,  \ \text{and} \ \forall\delta: 0<\delta\leq \delta_j.\notag
\end{align}
Let us now set
\begin{align}
\widetilde{H} \ &:= \ H_0+\sum_{j=1}^N\Psi_jH_j, \notag\\
\delta_0 \ &:= \ \min(\eps_0,\delta_1,\dots,\delta_N). \notag
\end{align}
So far, in view of \eqref{E:fact} and \eqref{E:leviPos}, we have accomplished the following:
\begin{enumerate}
\item[i)] $(P-\delta\widetilde{H})\in\psh(\CC)$, and $(P-\delta\widetilde{H})$ is {\em strictly}
plurisubharmonic on $\CC\setminus\bigcup_{j=1}^N L_j \ \forall\delta\in(0,\delta_0)$.
\item[ii)] $\{z:\widetilde{H}>0\}=
{\rm int}(\clcone_*)\bigcup(\bigcup_{j=1}^N(\scone{\zt_j}{2\sig{j}}\setminus L_j))$.
\end{enumerate}
All we now have to do is make a perturbation of $\widetilde{H}$ to get an $H$ that is
strictly positive where desired. To carry this out, let: 
\begin{align}
W_0 \ &:= \ S^3\bigcap\left({\rm int}(\clcone_*)\bigcup\left(\bigcup_{j=1}^N\scone{\zt_j}{2\sig{j}}
		\right)\right)^{{\sf C}}, \notag\\
W_1 \ &= \ \text{some small $S^3$-neighbourhood of $\overline{W_0}$ such that 
		$(\bigcup_{j=1}^NL_j)\bigcap W_1=\emptyset$}. \notag
\end{align}
Now let $\chi^*:S^3\lrarw [0,\alpha]$ be a smooth cut-off function on $S^3$ such that 
$\chi^*|_{W_0}\equiv\alpha$ and ${\rm supp}(\chi^*)\subset W_1$, where $\alpha>0$ is so
small that if we define
\[
H(z) \ := \ \widetilde{H}(z)+\|z\|^{2k}\chi^*\left(\frac{z}{\|z\|}\right),
\]
then --- in view of (i) above --- Part~(b) of this theorem follows without altering
the conclusion of (i) above when $\widetilde{H}$ is replaced by $H$. Hence (a) follows.
\end{proof}
\end{custom}
\smallskip

Next, we provide:

\begin{custom}\begin{proof}[{\bf The proof of Theorem~\ref{T:homobumpB}.}] Let us first
begin by defining $M:=$ the largest positive integer $\mu$ such that there exists some
$f\in\hol(\CC)$ and $f^\mu=\mathcal{H}$. Define $F$ by the relation $F^M=\mathcal{H}$.
Observe that the hypotheses of Theorem~\ref{T:homobumpB} continue to hold when 
$\mathcal{H}$ is replaced by $F$. 
\smallskip

\noindent{{\bf Step I.} {\em The function $F$ is a homogeneous polynomial}}

\noindent{This is a straightforward application of Lemma~\ref{L:homoImp}. Note that our
preliminary construction of $F$ is precisely the $F$ provided by Lemma~\ref{L:homoImp}
applied to $(m_1,m_2)=(2k,2k)$.}
\smallskip

\noindent{{\bf Step II.} {\em To show that $P$ is constant on each level-set of $F$}}

\noindent{First we note that, without loss of generality, we may assume that
$F|_{z_1=0}\not\equiv 0$. If not, we carry out the following change of coordinates
\begin{align}
Z_1 \ &:= \ z_1-\zt_0z_2 \notag\\
Z_2 \ &:= \ z_2, \notag
\end{align}
where $\zt_0\in\cplx\setminus\{0\}$ is so chosen that
$F|_{z_1=\zt_0z_2}\not\equiv 0$.
If we define
\[
\widetilde{P(}Z_1,Z_2) \ := \ P(Z_1+\zt_0Z_2,Z_2) \quad\text{and}\quad
\widetilde{F(}Z_1,Z_2) \ := \ F(Z_1+\zt_0Z_2,Z_2),
\]
then it is easy to check that
\begin{itemize}
\item $\widetilde{P{ }}$ is harmonic along the smooth part of each level curve of 
$\widetilde{F{ }}$; and
\item $\widetilde{F{ }}|_{Z_1=0}\not\equiv 0$.
\end{itemize}
Hence, we may as well assume that $F$ satisfies the desired condition. Then, by homogeneity
of $F$, $F(0, \ \bcdot)$ is non-constant. By the Fundamental Theorem of Algebra, then
\begin{equation}\label{E:nonempty}
\{z\in\CC:z_1=0\}\bigcap F^{-1}\{c\} \ \neq \ \emptyset\quad\forall c\in\cplx.
\end{equation}
Let us now assume that $P$ is non-constant on $F^{-1}\{c\}$ for some $c\in\cplx$.
The fact that $P$ being non-constant on $F^{-1}\{c\}$ is an open condition in $c\in\cplx$
implies, in conjunction with \eqref{E:nonempty}, that we can find a $c_0$ close to $c$ and
a $w_0\in\cplx$ such that
\begin{itemize}
\item $P$ is non-constant along $F^{-1}\{c_0\}$;
\item the point $q_0=(0,w_0)$ lies on $F^{-1}\{c_0\}$; and
\item $q_0$ is not a singular point of $F^{-1}\{c_0\}$.
\end{itemize}
Then, by construction, there exists an $\eps>0$ and a holomorphic map on the unit disc,
$\Psi=(\psi_1,\psi_2):\Dsc\lrarw\mathbb{B}(q_0;\eps)$ such that $\Psi(0)=q_0$ and $\Psi$ 
parametrises $F^{-1}\{c_0\}\bigcap\mathbb{B}(q_0;\eps)$.} 
\smallskip

Let us adopt the notations from the proof of Proposition~\ref{P:harmlines} and write
\[
P(z,w) \ = \ \sum_{j=M}^{2k}\sum_{\alpha+\beta=j \atop \mu+\nu=2k-j}
                C_{\alpha\beta\mu\nu}z^\alpha\zbar^\beta w^\mu\wbar^\nu.
\]
By hypothesis, the function
\[
v(\xi) \ := \ \sum_{j=M}^{2k}\sum_{\alpha+\beta=j \atop \mu+\nu=2k-j}
                C^j_{\alpha\beta\mu\nu}\psi_1(\xi)^\alpha
				\overline{\psi_1(}\xi)^\beta 
					\psi_2(\xi)^\mu
				\overline{\psi_2(}\xi)^\nu, \quad\xi\in\Dsc,
\]
is harmonic on the unit disc. Since, by hypothesis, $P$ has no pluriharmonic terms,
and the requirement of harmonicity forces $v$ to have only harmonic terms in its Taylor
expansion about $\xi=0$, we have:
\begin{multline}\label{E:harmKey}
v(\xi)  \\ 
= \ \sum_{j=M}^{2k}\sum_{\alpha+\beta=j \atop \mu+\nu=2k-j}
                C^j_{\alpha\beta\mu\nu}\left\{\psi_1(\xi)^\alpha\psi_2(\xi)^\mu
                                \overline{\psi_1(}0)^\beta\overline{\psi_2(}0)^\nu
					+ \psi_1(0)^\alpha\psi_2(0)^\mu
                                \overline{\psi_1(}\xi)^\beta\overline{\psi_2(}\xi)^\nu\right\}.
\end{multline}
However, recall that $\psi_1(0)=0$. In view of \eqref{E:harmKey}, this forces the conclusion
$v\equiv 0$. But this results in a contradiction because, by real-analyticity, 
$v=P\circ\Psi\equiv 0$ would force $P$ to vanish on $F^{-1}\{c_0\}$. This establishes
Step~II.
\smallskip

\noindent{{\bf Step III.} {\em The proof of Part~(1)}}

\noindent{Let us define the function $U:\cplx\lrarw\RR$ as
\[
U(c) \ := \ P(z_1^{(c)},z_2^{(c)}),
\]
where $(z_1^{(c)},z_2^{(c)})$ is any point lying in $F^{-1}\{c\}$. We would be done if
we could show that $U$ is real-analytic. Let us outline our strategy before tackling the details.
The strategy may be summarised as follows:
\begin{itemize}
\item[1)] We shall choose a complex line $\Lam_\tau:=\{(z_1,z_2=\tau z_1):z_1\in\cplx\}$
such that $F|_{\Lam_\tau}$ is non-constant. We can then show that for {\em almost every} 
$c_0\in\cplx$, we can find a function $u^{c_0}$ that is holomorphic in a neighbourhood 
$V^{c_0}\ni c_0$ and parametrises a designated root of the equation $F|_{\Lam_\tau}=c$ 
as $c$ varies through $V^{c_0}$. In other words: 
\begin{align}
u^{c_0}:c \longmapsto (u^{c_0}(c),\tau u^{c_0}(c)) &\in 
\bigcup_{\zt\in\cplx}\left(\Lam_\tau\cap F^{-1}\{\zt\}\right), 
\notag \\
(u^{c_0}(c),\tau u^{c_0}(c)) &\in F^{-1}\{c\} \quad \text{as $c$ varies through $V^{c_0}$.}\notag
\end{align}
\item[2)] Clearly, $U|_{V^{c_0}}=P(u^{c_0},\tau u^{c_0})$. As real-analyticity is a local
property, we would be done if the conclusions of (1) could be established in a neighbourhood of
{\em every} $c_0\in\cplx$. This can be achieved by repeating the above analysis on a different
complex line $\Lam_\eta\neq\Lam_\tau$. Subharmonicity would follow from a Levi-form calculation.
\end{itemize}
The details follow.}
\smallskip

Accordingly, choose any $\tau\in\cplx$ such that 
\[
F_\tau: z \ \longmapsto \ F(z,\tau z) \ \text{\em is a non-constant polynomial.}
\]
Recall that --- given a complex, univariate polynomial $p$ --- the map
\[
\Dis_p(c) \ := \ \text{the discriminant of the polynomial $p(z)-c$}
\]
has the following two properties:
\begin{itemize}
\item[i)] $\Dis_p(c)$ is a complex polynomial in $c$.
\item[ii)] $\Dis_p(c)=0 \ \iff$ the equation $p(z)-c=0$ has repeated roots.
\end{itemize}
The reader is referred to van der Waerden's book \cite{vanderWaerden:Av.I70} for an
exposition on the discriminant. With this information in mind, let us define
\[
\Dis_\tau(c) \ := \ \text{the discriminant of the polynomial $F_\tau(z)-c$.}
\]
Then, by (ii) above, $\Dis_\tau^{-1}\{0\}$ is a finite set, and if 
$c_0\in\cplx\setminus\Dis_\tau^{-1}\{0\}$, then there exists an open disc
$D(c_0;\delta)\subset\cplx\setminus\Dis_\tau^{-1}\{0\}$ such that the equation
$F_\tau(z)-c=0$ has ${\rm deg}(F_\tau)$ simple roots for each $c\in D(c_0;\delta)$.
In fact, we can find a $u^{c_0}\in\hol(D(c_0;\delta))$ such that 
\[
F_\tau(u^{c_0}(c))-c = 0 \quad\forall c\in D(c_0;\delta).
\]
Note that by the above equation and our hypothesis on $P$, we have
\begin{equation}\label{E:you}
U(c) \ = \ P(u^{c_0}(c),\tau u^{c_0}(c)) \quad\forall c\in D(c_0;\delta).
\end{equation}
Since $c_0$ was arbitrarily chosen from $\cplx\setminus\Dis_\tau^{-1}\{0\}$, and 
real-analyticity is a local property, we have just shown that 
$U\in\smoo^\omega(\cplx\setminus\Dis_\tau^{-1}\{0\})$. But we can now repeat the
above argument with some $\eta\neq\tau$, with the property that
\[
\left(\cplx\setminus\Dis_\tau^{-1}\{0\}\right)\bigcup\left(\cplx\setminus\Dis_\eta^{-1}\{0\}\right)
\ = \ \cplx,
\]
replacing $\tau$. We then get a version of equation~\eqref{E:you} with $\eta$ in place
of $\tau$. This establishes that $U\in\smoo^\omega(\cplx)$. By construction, $P=U\circ F$.
Given the homogeneity of $P$ and $F$ (from Step I), it is obvious that $U$ is homogeneous.
Now, performing a Levi-form computation, we get
\begin{align}\label{E:LeviformU}
\levi{P}(z_1,z_2;(V_1,V_2)) \ = \ \frac{1}{4}&\laplc U(F(z_1,z_2)) \\
        &\times \ (V_1 \;\; V_2)
                        \begin{pmatrix}
                        \ |\monodrv{1}F|^2  & \monodrv{1}F \ \monodrv{\overline{2}}\overline{F} \ \\
                        {} & {} \\
                        \ \monodrv{\overline{1}}\overline{F} \ \monodrv{2}F & |\monodrv{2}F|^2 \
                        \end{pmatrix}_{(z_1,z_2)}\begin{pmatrix}
                                                \overline{V}_1 \\
                                               {} \\
                                                \overline{V}_2
                                                \end{pmatrix}. \notag
\end{align}
Since $\levi{P}(z_1,z_2; \ \centerdot)$ must be positive semi-definite at every $(z_1,z_2)\in\CC$,
this forces the conclusion $\laplc U\geq 0$. Hence, $U$ is subharmonic, and Part~(1) is thus
established.
\smallskip
\pagebreak

\noindent{{\bf Step IV.} {\em The proof of Part~(2)}}

\noindent{Write $2\nu:={\rm deg}(U)$ (the degree of $U$ is even due
to subharmonicity). 
We apply Lemma~\ref{L:subhBump} to the subharmonic $U$ to obtain the smooth function $h$ that 
satisfies the conclusions of that lemma. Now define
\[
H(z_1,z_2) \ := \begin{cases}
			|F(z_1,z_2)|^{2\nu}h(\rg{F(z_1,z_2)}), 
				&\text{if $(z_1,z_2)\notin F^{-1}\{0\}$}, \\
			0, &\text{if $(z_1,z_2)\in F^{-1}\{0\}$},
			\end{cases}
\]
where $\rg{\bcdot}$ refers to any continuous branch of the argument.
Now, a Levi-form computation reveals that
\begin{align}\label{E:Leviform(U-H)}
\levi{(P-\delta H)}(z_1,z_2;(V_1,V_2)) \ = \ \frac{1}{4}&|F(z_1,z_2)|^{2(\nu-1)}
\laplc\left(U-\delta|\bcdot|^{2\nu}h\circ\rg{\bcdot}\right)(F(z_1,z_2))\\
        &\times \ (V_1 \;\; V_2)
                        \begin{pmatrix} 
                        \ |\monodrv{1}F|^2  & \monodrv{1}F \ \monodrv{\overline{2}}\overline{F} \ \\
                        {} & {} \\
                        \ \monodrv{\overline{1}}\overline{F} \ \monodrv{2}F & |\monodrv{2}F|^2 \
                        \end{pmatrix}_{(z_1,z_2)}\begin{pmatrix}
                                                \overline{V}_1 \\
                                               {} \\
                                                \overline{V}_2
                                                \end{pmatrix}. \notag 
\end{align}
In view of Lemma~\ref{L:subhBump}/(b), $(P-\delta H)$ is clearly plurisubharmonic 
$\forall\delta\in(0,1)$. Furthermore, note that, by the properties of $h$, 
\[
H \ \geq \ 0 \quad\text{and $ \; H(z_1,z_2)=0 \ \iff \ F(z_1,z_2)=0$}.
\]
This establishes Part~(2).}
\end{proof}
\end{custom}
\medskip

\section{The proof of Theorem~\ref{T:2homo}}\label{S:proofs2homo}

To avoid confusion resulting from subscripts, we shall write $z:=z_1$ and $w:=z_2$.
We shall also adopt several of the conventions and facts that feature in the proof of
Proposition~\ref{P:harmlines}. Accordingly, let us write
\begin{equation}\label{E:expan}
Q(z,w) \ = \ \sum_{\alpha,\beta\geq 0}
                C_{\alpha\beta}z^\alpha\zbar^{2p-\alpha} w^\beta\wbar^{2q-\beta},
\end{equation}
As before, let us consider the complex lines $L^\zt := \{(z=\zt w,w):w\in\cplx\}$ and examine
$Q|_{L^\zt}$. As in Proposition~\ref{P:harmlines}, we write
\[
Q(\zt w,w) \ = \ \sum_{m+n=2(p+q)}\left\{\sum_{\alpha+\beta=m}
                C_{\alpha\beta}\zt^\alpha\ztbar^{2p-\alpha}\right\}w^m\wbar^n \ \equiv \
                \sum_{m+n=2(p+q)}\phi_{mn}(\zt)w^m\wbar^n.
\]
Recall from Proposition~\ref{P:harmlines} that $\phi_{p+q,p+q}$ is a subharmonic function,
$\phi_{p+q,p+q}\geq 0$, and $\phi_{p+q,p+q}\not\equiv 0$. All of this implies that (note that  
$\phi_{p+q,p+q}$ is homogeneous)
\[
0 \ < \ \int_0^{2\pi}\phi_{p+q,p+q}(e^{i\tht})d\tht \ = \ C_{pq}.
\]
We have just concluded that in the expansion \ref{E:expan}, the term $|z|^{2p}|w|^{2q}$ occurs 
with a positive coefficient. Let us thus decompose $Q$ as
\[
Q(z,w) \ = \ C_{pq}|z|^{2p}|w|^{2q} + R(z,w) \ \equiv \ A(z,w)+R(z,w).
\]
Note that $A$ is harmonic along the varieties $V_c:=\{(z,w)\in\CC:z^pw^q=c\}$. Since, generically
in $V_c$, $T^{\cplx}_{(z,w)}(V_c)=\spn_{\cplx}[(qz,-pw)]$, we have
\begin{equation}\label{E:flat}
\levi{A}((z,w); v) \ = \ 0 \quad\forall v\in\spn_{\cplx}[(qz,-pw)] \ \text{and} \
                                \forall (z,w)\in\CC.
\end{equation}
In other words, equation~\eqref{E:flat} holds true for every $(z,w)\in\CC$, independent of the
variety $V_c$ to which $(z,w)$ belongs. By plurisubharmonicity of $Q$, we infer that
\begin{equation}\label{E:pshR}
\levi{R}((z,w); (qz,-pw)) \ \geq \ 0 \quad\forall(z,w)\in\CC.
\end{equation}
\smallskip
                
Let us now write $z=re^{i\tht}$ and $w=se^{i\tau}$. In this notation, we get
\begin{align}
R(z,w) \ &= \ |z|^{2p}|w|^{2q}T(\tht,\tau), \notag \\
\text{where}\quad T(\tht,\tau) \ &:= \
\sum_{(\alpha,\beta)\neq(p,q)}C_{\alpha\beta}e^{i(2\alpha-2p)\tht}e^{i(2\beta-2q)\tau}.\label{E:tee}
\end{align}
It is a routine matter to check that
\begin{align}
4 \ \mixderiv{z}{z}{{}} \ &= \ \deriv{2}{{}}{r^2}+\frac{1}{r}\deriv{{}}{{}}{r}+
                                \frac{1}{r^2}\deriv{2}{{}}{\tht^2} \notag \\
4 \ \mixderiv{w}{w}{{}} \ &= \ \deriv{2}{{}}{s^2}+\frac{1}{s}\deriv{{}}{{}}{s}+
                                \frac{1}{s^2}\deriv{2}{{}}{\tau^2} \notag \\
4 \ \mixderiv{z}{w} \ &= \ e^{i(\tau-\tht)}\left[\mixRderiv{r}{s}{{}}+
                        \frac{1}{rs} \ \mixRderiv{\tht}{\tau}{{}} - i\left\{
                        \frac{1}{r} \ \mixRderiv{\tht}{s}{{}}-\frac{1}{s} \ \mixRderiv{r}{\tau}{{}}
                        \right\}\right] \notag
\end{align}
Using these differential operators in the inequality \eqref{E:pshR} gives us
\begin{align} 
r&^{2p}s^{2q} \notag\\
&\times[q \; -p]
                        \begin{bmatrix}
                        \ 4p^2T+T_{\tht\tht}  & 4pqT+T_{\tht\tau}+i(2pT_\tau-2qT_\tht) \ \\
                        {} & {} \\
                        \ 4pqT+T_{\tht\tau}-i(2pT_\tau-2qT_\tht) & 4q^2T+T_{\tau\tau} \
                        \end{bmatrix}\begin{bmatrix}
                                                q \\
                                               {} \\
                                               -p
                                                \end{bmatrix} \notag\\
&\geq \ 0 \quad \forall r,s\geq 0, \ \forall(\tht,\tau)\in[-\pi,\pi)\times[-\pi,\pi).\notag
\end{align}
Simplifying the above gives us
\begin{align}
(q^2T_{\tht\tht}&-2pqT_{\tht\tau}+p^2T_{\tau\tau})(\tht,\tau) \notag \\
&= \ \left(q\deriv{{}}{{}}{\tht}-p\deriv{{}}{{}}{\tau}\right)
                \left(q\deriv{{}}{{}}{\tht}-p\deriv{{}}{{}}{\tau}\right)T(\tht,\tau) \ \geq \ 0
\quad\forall(\tht,\tau)\in[-\pi,\pi)\times[-\pi,\pi).\notag
\end{align}
The above inequality simply tells us that for every line on the $\tht\tau$-plane having
tangent vector $(q,-p)$, i.e. for every line $\ell^C:=\{(\tht,\tau)\in\RR^2:p\tht+q\tau=C\}$,
\[
T|_{\ell^C} \ \text{\em is convex for each $C\in\RR$}.
\]
However, the definition of $T$ in \eqref{E:tee} above reveals that $T$ is a real-analytic function
as well as $2\pi$-periodic. For such a $T$ to be convex, necessarily
\[
T|_{\ell^C} \ \equiv \ \text{const.} \;\; \text{\em for each line $\ell^C\subset\RR^2$}.
\]
Hence, $T$ must have the form $T(\tht,\tau)=G(p\tht+q\tau)$, where $G$ is a periodic function.
This means that $T$ must have the form
\begin{equation}\label{E:teekey}
T(\tht,\tau) \ = \ \sum_{M\in\mathfrak{F}}C_Me^{iM(p\tht+q\tau)}, \quad C_M\neq 0 \
                                                                \forall M\in\mathfrak{F},
\end{equation}
where $\mathfrak{F}\subset\zahl$ is a finite subset of integers. Comparing \eqref{E:tee} with
\eqref{E:teekey}, we infer that
\begin{equation}\label{E:indices}
C_{\alpha\beta}\neq 0 \ \Longrightarrow \ \exists M\in\mathfrak{F} \ \text{such that} \
                \frac{2\alpha-2p}{p} = M = \frac{2\beta-2q}{q}.
\end{equation}
If we define $d:=\gcd(\alpha:C_{\alpha\beta}\neq 0)$, we can immediately infer the following
facts:
\begin{itemize}
\item Since $R$ is real-valued,
$C_{\alpha,\beta}\neq 0\Longrightarrow C_{2p-\alpha,2q-\beta}\neq 0$ --- whence  
$d|(2p-\alpha)$ for any $\alpha$ such that $C_{\alpha\beta}\neq 0$; and
\item Owing to \eqref{E:indices}
\[
\beta_0\in\{\beta:C_{\alpha\beta}\neq 0\} \ \iff \ \exists\alpha_0\in
                        \{\alpha:C_{\alpha\beta}\neq 0\} \ \text{such that $\beta_0=q\alpha_0/p$}.
\]
\end{itemize}
Therefore $D:=\gcd(\beta:C_{\alpha\beta}\neq 0)=qd/p$, whence we can find a real-valued polynomial
$r$ that is homogeneous of degree $2p/d$ such that $R(z,w)=r(z^dw^D)$. Now set
\[
U(\xi) \ := \ C_{pq}|\xi|^{2p/d} + r(\xi) \quad\forall\xi\in\cplx.
\]
Clearly, $Q(z,w)=U(z^dw^D)$. We compute the Levi-form of $Q$ one last time. In the process, we
get
\begin{align}
\levi{Q}&((z,w); v) \notag \\
&= \ U_{\xi\bar{\xi}}(z^dw^D)|z|^{2(d-1)}|w|^{2(D-1)}
                                \times(v_1 \;\; v_2)
                                \begin{pmatrix}
                                \ d^2|w|^2 & Dd\zbar w \ \\
                                {} & {} \\
                                \ Ddz\wbar & D^2|z|^2\
                                \end{pmatrix} \begin{pmatrix}
                                                \ \overline{v_1} \ \\
                                                {} \\
                                                \ \overline{v_2} \
                                                \end{pmatrix} \ \geq \ 0 \notag\\
& \ \forall(z,w)\in\CC, \ \forall v\in\CC.\notag
\end{align}
Hence,
\[
(\laplc U)(z^dw^D) \ \geq \ 0 \quad\forall (z,w)\in\CC.
\]
Since $z^dw^D$ attains every value in $\cplx$ as $(z,w)$ varies through $\CC$, we infer that
$\laplc U\geq 0$, i.e. that $U$ is subharmonic. This final fact completes the proof. \qed
\medskip

\section{Proofs of the Main Theorems}

We are now ready to provide a proof of Main Theorem~\ref{MT:homobump}. The basic idea --- i.e. of 
examining the pullback of $P$ by a suitable proper holomorphic map such that the pullback is
homogeneous --- is a simple one. The following argument provides the details. 
\smallskip

\begin{custom}\begin{proof}[{\bf The proof of Main Theorem~\ref{MT:homobump}.}] 
Define $K:=\lcm(m_1,m_2)$ (i.e. the least common multiple of $m_1$ and $m_2$) and write 
$\sigma_j:=K/m_j, \ j=1,2.$ 
Define the proper holomorphic map $\Psi:\CC\lrarw\CC$ by $\Psi(z_1,z_2):=(z_1^{\sigma_1},z_2^{\sigma_2})$,
and write $Q=P\circ\Psi$. Fix a point $(z_1^0,z_2^0)\in\CC\setminus\{(z_1,z_2):z_1z_2=0\}$. Then,
there exist neighbourhoods $U^j\ni z_j^0$ such that the functions $(\xi\mapsto\xi^{\sigma_j})|_{U^j}$ are
injective, $j=1,2$. Therefore $\Psi|_{U^1\times U^2}$ is a biholomorphism, whence 
$Q|_{U^1\times U^2}\in\psh(U_1\times U_2)$. Since plurisubharmonicity is a local property, we 
infer that $Q\in\psh(\CC\setminus\{(z_1,z_2):z_1z_2=0\})$. Finally, since $Q$ is smooth and 
$\{(z_1,z_2):z_1z_2=0\}$ is a pluripolar set, we infer that $Q\in\psh(\CC)$. Furthermore:
\[
Q(t^{1/K}z_1,t^{1/K}z_2) \ = \ P((t^{1/K}z_1)^{\sigma_1},(t^{1/K}z_2)^{\sigma_2}) 
\ = \ tQ(z_1,z_2) \;\; \forall t>0, 
\]
whence $Q$ is a plurisubharmonic polynomial that is homogeneous of degree $K$. By hypothesis, $Q$
has no pluriharmonic terms. Furthermore, we observe that
\[
\frac{m_1}{\gcd(m_1,m_2)} \ = \ \sigma_2, \quad \frac{m_2}{\gcd(m_1,m_2)} \ = \ \sigma_1,
\]
and hence note that for any $\zt\in\cplx$ such that 
$\{(w_1,w_2):w_1^{\sigma_2}=\zt w_2^{\sigma_1}\}\in\exepc(P)$, $Q$ is forced to be harmonic along
each of the complex lines that make up the set
\[
\cplxlns(\zt) \ := \ \bigcup_{l=0}^{\sigma_1\sigma_2-1}
\left\{(z_1,z_2):z_1=|\zt|^{1/\sigma_1\sigma_2}\exp\left(\frac{2\pi il+i{\sf 
Arg}(\zt)}
{\sigma_1\sigma_2}\right)z_2\right\}
\]
(here ${\sf Arg}$ denotes some branch of the argument). But since, by Proposition~\ref{P:harmlines}, 
there are only finitely many complex lines along which $Q$ can be harmonic, this implies Part~(1) of
our theorem.
\smallskip

Now, let $\zt_1,\dots,\zt_N$ be as in Part~(1) of the statement of this theorem. Suppose there is some
$z^0\neq 0$ and a germ of a complex variety $V^0$ at $z^0$ such that
\begin{align}\label{E:assumHarm}
V^0 \ \nsubseteq \ \cplxlns(\zt_j) \;\;  &\forall j=1,\dots, N, \ \text{and,} \\
V^0 \ \nsubseteq \ \{(z_1,z_2):z_2=0\} \;\; &\text{if $\{(z_1,z_2):z_2=0\}\in\exepc(Q)$.}\notag
\end{align}
By assumption, we have the following
situations for $V^0$:
\smallskip

\noindent{{\bf Case (i)} $V^0\nsubseteq\{(z_1,z_2):z_1z_2=0\}$}

\noindent{In this case, we can select a domain $\OM\subset\CC$ such that
\begin{itemize}
\item $\OM\bigcap\{(z_1,z_2):z_1z_2=0\}=\emptyset$ and $\Psi|_\OM$ is a biholomorphism;
\item $\OM\bigcap V^0$ is a smooth subvariety of $\OM$; and
\item we can find a regular parametrisation $\varphi=(\varphi_1,\varphi_2):
\mathbb{D}\lrarw(V^0\bigcap\OM)$
of $V^0\bigcap\OM$.
\end{itemize}
Here, $\mathbb{D}$ denotes the open unit disc in $\cplx$. We now compute that
\begin{align}
\laplc(P\circ\Psi\circ\varphi)(\xi) \ &= \ 
	\levi{P}(\Psi\circ\varphi(\xi);\Psi_*\left|_{\varphi(\xi)}\right.(\varphi^\prime(\xi))) \notag\\
&= \ \levi{Q}\left(\varphi(\xi);\left(\Psi^{-1}\right)_*\left|_{\Psi\circ\varphi(\xi)}\right.
\left\{
\Psi_*\left|_{\varphi(\xi)}\right.(\varphi^\prime(\xi))\right\}\right) \notag\\
&= \ \laplc\left(Q|_{V^0\cap\OM}\right) \ = \ 0,
\end{align}
i.e., we conclude that $P$ is harmonic along $\Psi(V^0\bigcap\OM)$. Yet, by the assumption
\eqref{E:assumHarm}, $\Psi(V^0\bigcap\OM)$ is not contained in any curve belonging to
$\exepc(P)$. But this contradicts the hypothesis that $P$ has Property~(A), whence this
case cannot arise.}
\smallskip

\noindent{{\bf Case (ii)} {\em Either $V^0\subset\{(z_1,z_2):z_1=0\}$ or 
$V^0\subset\{(z_1,z_2):z_2=0\}$.}}

\noindent{A {\em much} simpler variant of the above argument shows us that
these cases will not arise depending on whether $\{(z_1,z_2):z_1=0\}\notin\exepc(Q)$,
or $\{(z_1,z_2):z_2=0\}\notin\exepc(Q)$, respectively.} 
\smallskip

We have therefore established the following fact:
\begin{align}\label{E:QFact}
Q \ \text{\em is a plurisubharmonic polynomial that is homogeneous of degree $K$, and}\\
\mathfrak{C}(Q) \ = \ \text{\em the set of all complex lines belonging to 
$\exepc(Q)$.}\notag
\end{align}
In fact, in view of the above fact, it is completely routine to infer that
\[
\text{\em $P$ has Property~(A)} \ \Longrightarrow \ \text{\em $Q$ has Property~(A)}.
\]
\smallskip

Now consider the unitary transformations $R^{lm}:(z_1,z_2)\lrarw (e^{2\pi il/\sigma_1}z_1,e^{2\pi 
im/\sigma_2}z_2)$ and compute: 
\begin{align} 
\levi{Q}(R^{lm}(z);R^{lm}(V)) &= \levi{P\circ\Psi}\left(R^{lm}(z);
		\Psi_*\left|_{R^{lm}(z))}\right.(R^{ln}V)\right) \notag \\
	&=\ \left(\sigma_1z_1^{\sigma_1-1}e^{-2\pi il/\sigma_1}(R^{lm}V)_1 \;\;\;
	\sigma_2z_2^{\sigma_2-1}e^{-2\pi im/\sigma_2}(R^{lm}V)_2\right) \notag\\
&\qquad\quad\times \hess(P)|_{\Psi(R^{lm}(z))} \ 
	\begin{pmatrix}
	\sigma_1\zbar_1^{\sigma_1-1}e^{2\pi il/\sigma_1}\overline{(R^{lm}V)}_1 \\
	\sigma_2\zbar_2^{\sigma_2-1}e^{2\pi im/\sigma_2}\overline{(R^{lm}V)}_2
	\end{pmatrix} \notag \\
&= \ \levi{P}\left(\Psi(z);\Psi_*\left|_z\right.(V)\right) \notag \\
&= \ \levi{Q}(z;V). \label{E:levimatch}
\end{align}
So, if we define $\mathfrak{N}_Q(z)$ to be the null-space of $\levi{Q}(z; \ \bcdot)$, then
the computation \eqref{E:levimatch} reveals that
\begin{multline}\label{E:rot}
z\in\LeviDeg{Q} \ \text{and} \ V\in\mathfrak{N}_Q(z) \\
\iff \ R^{lm}(z)\in\LeviDeg{Q} \ \text{and} \ R^{lm}(V)\in\mathfrak{N}_Q(R^{lm}(z)) \ 
\forall l,m:1\leq l\leq\sigma_1, \ 1\leq m\leq\sigma_2.
\end{multline}  
In particular \eqref{E:rot} implies that if, in the notation borrowed from the proof of 
Theorem~\ref{T:homobump}, we can find a constant $\alpha_j>0$ and a $H_j\in\smoo^\infty(\CC)$ that
is homogeneous of degree $K$ such that
\begin{itemize}
\item $H_j=0$ on $L_j:=\{z:z_1=|\zt_j|^{1/\sigma_1\sigma_2}\exp(i{\sf Arg}(\zt_j)/\sigma_1\sigma_2)z_2\}$,
\item $(Q-\delta H_j)\in{\sf spsh}[\scone{|\zt_j|^{1/\sigma_1\sigma_2}\exp(i{\sf Arg}(\zt_j)/\sigma_1\sigma_2)}
{\alpha_j}\setminus L_j]$ for each $\delta:0<\delta\leq 1$,
\end{itemize}
then the above remains true with
\begin{itemize}
\item $H_j$ replaced by $H^{(lm)}_j:=H_j\circ(R^{lm})^{-1}$,
\item $L_j$ replaced by
\[
L^{(lm)}_j \ := \ \left\{z:z_1=|\zt_j|^{1/\sigma_1\sigma_2}\exp\left(\frac{2\pi i(\sigma_1m-\sigma_2l)
+i{\sf Arg}(\zt_j)}{\sigma_1\sigma_2}\right)z_2\right\},
\]
\item the cone $\scone{|\zt_j|^{1/\sigma_1\sigma_2}\exp(i{\sf Arg}(\zt_j)/\sigma_1\sigma_2)}{\alpha_j}$
is replaced by its image under $R^{lm}$,
\end{itemize}
for any $l,m:1\leq 1\leq\sigma_1, \ 1\leq m\leq\sigma_2.$
\smallskip

Since $Q$ has Property~(A), Result~\ref{R:noell} is applicable. A careful examination of its proof
reveals that Noell's construction of the bumping is local. Then, in view of \eqref{E:rot} and the
preceding discussion, we can --- by selecting our cut-off functions in Theorem~\ref{T:homobump} 
to be equivariant with respect to $R^{lm}$ --- construct our bumping 
$H$ (of the polynomial $Q$) to have the property
\begin{align}\label{E:propH}
H(z_1,z_2) \ = \ H(e^{2\pi il/\sigma_1}z_1,e^{2\pi im/\sigma_2}z_2) \ &\forall z\in\CC \ \text{and} \\
		&\forall l,m:1\leq 1\leq\sigma_1, \ 1\leq m\leq\sigma_2.\notag
\end{align}
Now define
\begin{multline}
G(z) \\
:= \ \frac{1}{\sigma_1\sigma_2}\sum_{j=1}^{\sigma_1}\sum_{k=1}^{\sigma_2}
		H\left(|z_1|^{1/\sigma_1}\exp\left(\frac{2\pi ij+i{\sf Arg}(z_1)}{\sigma_1}\right),
	|z_2|^{1/\sigma_2}\exp\left(\frac{2\pi ik+i{\sf Arg}(z_2)}{\sigma_2}\right)\right).\notag
\end{multline}
Observe, however, that by the definition of $Q$, $P$ satisfies
\begin{multline}\label{E:defP}
P(z) \\ 
= \ \frac{1}{\sigma_1\sigma_2}\sum_{j=1}^{\sigma_1}\sum_{k=1}^{\sigma_2}
                Q\left(|z_1|^{1/\sigma_1}\exp\left(\frac{2\pi ij+i{\sf Arg}(z_1)}{\sigma_1}\right),
        |z_2|^{1/\sigma_2}\exp\left(\frac{2\pi ik+i{\sf Arg}(z_2)}{\sigma_2}\right)\right).
\end{multline}
Let $\delta_0>0$ be as given by Theorem~\ref{T:homobump} applied to $Q$. Now, as in the beginning
of this proof, fix $(z_1^0,z_2^0)\in\CC\setminus\{(z_1,z_2):z_1z_2=0\}$ and let 
$U^j\ni z_j^0, \ j=1,2$, be neighbourhooods such that $(\xi\mapsto\xi^{\sigma_j})|_{U^j}$ are
injective and such that $(U^1\times U^2)\bigcap\{(z_1,z_2):z_1z_2=0\}=\emptyset$. Write
$V^1\times V^2:=\Psi(U^1\times U^2)$. Note that
by the definition of $G$, and by \eqref{E:propH} and \eqref{E:defP}, we have
\[
(P-\delta G)|_{V^1\times V^2} \ = \ (Q-\delta H)\circ\left(\Psi|_{U^1\times U^2}\right)^{-1}.
\]
Then, whenever $0<\delta\leq\delta_0$, we have the Levi-form computation
\begin{align}
\levi{(P-\delta G)}(w;V) \ &=  \ \levi{(Q-\delta H)}\left((\Psi|_{U^1\times U^2})^{-1}(w);
					\left((\Psi|_{U^1\times U^2})^{-1}\right)_*\left|_w\right.
										(V)\right)\notag\\
&\geq \ 0 \quad\forall w\in V^1\times V^2, \ \forall V\in\CC. \notag
\end{align}
The above argument establishes that whenever $0<\delta\leq\delta_0$, 
$(P-\delta G)\in\psh[\CC\setminus\{w:w_1w_2=0\}]$. Since $(P-\delta G)\in\smoo^2(\CC)$, we infer
that $(P-\delta G)\in\psh(\CC)$ by exactly the same argument as in the first paragraph of this
proof. Finally, given the relationship between the sets $\bigcup_{j=1}^N\cplxlns(\zt_j)$ and 
$\bigcup_{C\in\exepc(P)}C$, Part~(2) follows.
\end{proof} 
\end{custom} 
\smallskip

\begin{custom}\begin{proof}[{\bf The proof of Main Theorem~\ref{MT:homobumpB}.}]
We will re-use the ideas in the preceding proof, but we shall be brief.
Let $\mathcal{H}$ be as in the hypothesis of the theorem and, as before, define $M:=$ the 
largest positive integer $\mu$ such that there exists some
$g\in\hol(\CC)$ and $g^\mu=\mathcal{H}$. Define $F$ by the relation $F^M=\mathcal{H}$.
Our hypotheses continue to hold when
$\mathcal{H}$ is replaced by $F$ and, by Lemma~\ref{L:homoImp}, $F$ is $(m_1,m_2)$-homogeneous.
Let $K$ (i.e. the least common multiple of $m_1$ and $m_2$), $\sigma_1$, $\sigma_2$, $Q$, and the
proper holomorphic map $\Psi:\CC\lrarw\CC$ be {\em exactly} as in the proof of 
Main Theorem~\ref{MT:homobump}. We recall, in particular, that:
\begin{align}
\Psi(z_1,z_2)\ &:= \ (z_1^{\sigma_1},z_2^{\sigma_2}), \notag \\
Q \ &:= \ P\circ\Psi. \notag 
\end{align}
And as before, $Q$ is homogeneous of degree $K$.
\smallskip

Furthermore, if we define $f:=F\circ\Psi$, we get:
\begin{align}
(z_1,z_2) &\in f^{-1}\{c\} \notag\\
\Rightarrow \;\; \Psi^{-1}\{(z_1,z_2)\} &\subset F^{-1}\{c\} \notag \\
\Rightarrow & \;\; \text{$Q$ is harmonic along the smooth part of $f^{-1}\{c\}$.}\notag
\end{align}
Furthermore, we leave the reader to verify that $Q$ has no pluriharmonic terms. Therefore,
by applying Theorem~\ref{T:homobumpB}, we obtain a homogeneous, subharmonic polynomial 
$U$ such that $Q=U\circ f$. Let us write $2\nu:={\rm deg}(U)$. Recall that by applying
Lemma~\ref{L:subhBump} to this $U$ to obtain the $h$ as stated in that lemma, and defining
$\widetilde{\eh}(z):=|z|^{{\rm deg}(U)}h(\rg{z})$, we get
\begin{align}\label{E:homopsh}
P-\delta(\widetilde{\eh}\circ f) \ \in \ \psh(\CC) \quad\forall\delta: 0<\delta\leq 1.
\end{align}
(The above $\widetilde{\eh}\circ f$ is precisely what we called $H$ in Theorem~\ref{T:homobumpB}.)
It is clear that $P=U\circ F$. This establishes Part~(1) of the theorem. The fact that $F$ is
$(m_1,m_2)$-homogeneous {\em with weight $1/2\nu$} follows from degree considerations; since
$P$ is $(m_1,m_2)$-homogeneous, $F$ must be homogeneous with weight $1/{\rm deg}(U)=1/2\nu$.
\smallskip

Now define
\[
G(z_1,z_2) \ := \ \widetilde{\eh}\circ F(z_1,z_2) \quad\forall (z_1,z_2)\in\CC.
\]
By \eqref{E:homopsh}, and by a repetition of the argument in the second half of the last
paragraph of the proof of Main Theorem~\ref{MT:homobump}, we infer
\[
Q-\delta G \ \in \ \psh(\CC) \quad\forall\delta: 0<\delta\leq 1.
\]
\end{proof}
\end{custom}

\end{document}